%% file: ST.tex
\documentclass[12pt,reqno]{amsart}
\usepackage{amssymb,amscd}
\usepackage[alphabetic]{amsrefs}
\usepackage[all]{xy}
\usepackage[headings]{fullpage}
\usepackage[T1]{fontenc}
\usepackage{libertine}

\include{formatting}
\include{macros}

\DeclareMathOperator{\BS}{BS}
\newcommand{\Frr}{{\mathbb{F}_r}} 

\newcommand{\gP}{\mathfrak{P}}
\newcommand{\teich}{\mathbf{t}}
\newcommand{\II}{\mathbf{II}}
\newcommand{\trivcar}{\mathbf{1}}
\newcommand{\hght}{ht}
\newcommand{\ee}{{e}}
\renewcommand{\setminus}{\smallsetminus}

\theoremstyle{plain}
\newtheorem{lemmas}[subsection]{Lemma} 

\begin{document}
\title[Sextic Artin--Schreier twists]
{On the arithmetic of a family of \\ 
twisted constant elliptic curves}

\author{Richard \textsc{Griffon}}
\address{Departement Mathematik und Info., Universit\"at Basel,  Spiegelgasse 1, CH-4051 Basel, Switzerland}
\email{richard.griffon@unibas.ch}

\author{Douglas \textsc{Ulmer}}
\address{Department of Mathematics, University of Arizona, Tucson, AZ 85721 USA}
\email{ulmer@math.arizona.edu}

\date{\today}

\begin{abstract}
Let $\F_r$ be a finite field of characteristic $p>3$.
For any power $q$ of $p$, consider the elliptic curve $E=E_{q,r}$ defined by $y^2=x^3 + t^q -t$ over $K=\F_r(t)$.
We describe several arithmetic invariants of $E$ such as the rank of its Mordell--Weil group $E(K)$, the size of its N\'eron--Tate regulator $\Reg(E)$, and the order of its Tate--Shafarevich group $\sha(E)$ (which we prove is finite).
These invariants have radically different behaviours depending on the congruence class of $p$ modulo 6. 
For instance  $\sha(E)$ either has trivial $p$-part or is a $p$-group.
On the other hand, we show that the product $|\sha(E)|\Reg(E)$ has size comparable to $r^{q/6}$ as $q\to\infty$, regardless of $p\pmod{6}$.
Our approach relies on the BSD conjecture, an explicit expression for the $L$-function of $E$, and a geometric analysis of the N\'eron model of $E$. 
\end{abstract}

\subjclass[2010]{
Primary 
 11G05, 
 14J27; 
Secondary 
 11G40, 
 11G99, 
 14G10, 
 14G99. 
}

\maketitle

\section{Introduction}
For a prime $p>3$, and powers $q$ and $r$ of $p$, we study the elliptic curve
\begin{equation*}
E: \quad y^2=x^3+t^q-t    
\end{equation*}
over the rational function field $K=\Frr(t)$.  We are interested in
the Mordell--Weil group $E(K)$, its regulator $\Reg(E)$, and the
Tate--Shafarevich group $\sha(E)$ of $E$.  By old results of Tate \cite{Tate66b} and
Milne \cite{Milne75}, $\sha(E)$ is finite and the conjecture of Birch and
Swinnerton-Dyer holds for $E$.

One of our main results says that $\Reg(E)|\sha(E)|$ is an integer
comparable in archimedean size to $r^{q/6}$ when $r$ is fixed and $q$
tends to $\infty$.  (See Theorem~\ref{thm:BS} below for the precise
statement.)  On the other hand, we will show that if
$p\equiv1\pmod 6$, then $E(K)=0$, $\Reg(E)=1$, and $|\sha(E)|$ is a
$p$-adic unit; and that if $p\equiv-1\pmod6$ and $\Frr$ is
sufficiently large, then $E(K)$ has rank $2(q-1)$, $\Reg(E)|\sha(E)|$
is a power of $p$, and $\sha(E)$ is a $p$-group
(Propositions~\ref{prop:ord-L-p=1(6)} and \ref{prop:ord-L-p=-1(6)},
and Corollary~\ref{cor:p-parts-via-L}).  These results show in
particular that the archimedean and $p$-adic sizes of
$\Reg(E)|\sha(E)|$ are independent---in our examples,
$\Reg(E)|\sha(E)|$ is large in the archimedean metric, whereas it may
be a $p$-adic unit or divisible by a large power of $p$.

To prove these results, we combine an analytic analysis of the special
value $L^*(E)$, the Birch and Swinnerton-Dyer (BSD) formula, and an
algebraic analysis of $\sha(E)$.  We are able to deduce the BSD
formula and analyze $\sha(E)$ by using the fact that the N\'eron model
$\EE\to\P^1$ of $E$ is birational to the quotient of a product of
curves by a finite group.  In fact, $\EE$ has three distinct such
presentations, and each is convenient for some aspect of our study.

The plan of the paper is as follows: In the next section, we gather
the basic definitions and present a few preliminary results about $E$.
In Section~\ref{s:exp-sums}, we recall standard results about Gauss
and Jacobi sums and use them in Section~\ref{s:L-elem} to give an
elementary calculation of the Hasse--Weil $L$-function of $E$.  In
Section~\ref{s:curves}, we prove results about the geometry and
cohomology of certain curves over $\Frr$ which are used in
Section~\ref{s:domination} to show that the N\'eron model of $E$ is
dominated by a product of curves (in multiple ways).  In Section~\ref{s:L-cohom}, 
we use these dominations to give alternate calculations of the $L$-function.
In Section~\ref{s:BSD}, we
apply the BSD conjecture to study the rank of $E(K)$, and in
Section~\ref{s:BSD-p} we study the $p$-adic size of the special value
and the order of $\sha(E)$ using the BSD formula.
Section~\ref{s:sha-alg} reproves our results about $\sha(E)$ by a
direct, algebraic approach, i.e., independently of the BSD formula.
Finally, in Section~\ref{s:BS}, we study the archimedean size of the special
value and the ``Brauer--Siegel ratio'' of Hindry. 

The following table summarizes our main results:

\renewcommand{\arraystretch}{1.4}
\begin{center}
\begin{tabular}{c | c | c |}
& $p\equiv 1\bmod{6}$ & $p\equiv-1\bmod{6}$\\ \hline
$E(K)_{\mathrm{tors}}$ & \multicolumn{2}{|c|}{$\cong\{0\}$} \\
& \multicolumn{2}{|c|}{(Proposition~\ref{prop:Tawagawa-torsion}(2))} \\ \hline
BSD conjecture &  \multicolumn{2}{|c|}{holds for $E$} \\ 
& \multicolumn{2}{|c|}{(Theorem~\ref{thm:BSD})}  \\ \hline 
$\rk E(K)$ & $= 0$ & $= 2(q-1)$ for $\Frr$ large enough \\
& (Proposition~\ref{prop:ord-L-p=1(6)}(3)) &  (Proposition~\ref{prop:ord-L-p=-1(6)}(3)) \\ \hline
$\mathrm{Reg}(E)$ & $= 1$ & is a power of $p$ for $\Frr$ large enough \\ 
& (Proposition~\ref{prop:ord-L-p=1(6)}(4)) & (Corollary~\ref{cor:p-parts-via-L}(3)) \\ \hline
$\sha(E)$ & has trivial $p$-part & is a $p$-group \\ 
& (Proposition~\ref{prop:sha-alg}(1)) &  (Corollary~\ref{cor:p-parts-via-L}(3))\\\hline
$\dim\sha(E)$ & $= 0$ & $= \lfloor q/6\rfloor$  \\
& (Corollary~\ref{cor:dimsha-via-L}(1)) & (Corollary~\ref{cor:dimsha-via-L}(2))\\ \hline
$\lim_{q\to\infty}\BS(E)$ & \multicolumn{2}{|c|}{$= 1$} \\
& \multicolumn{2}{|c|}{(Theorem~\ref{thm:BS})} \\ \hline
$|\sha(E)|\mathrm{Reg}(E)$ & $\geq r^{\lfloor{q}/{6}\rfloor (1+o(1))}$ as $q\to\infty$  & $= r^{\lfloor q/6 \rfloor}$ for $\Frr$ large enough \\
& (Corollary~\ref{cor:BS-p=1(6)}) & (Corollary~\ref{cor:p-parts-via-L}(3)) \\ \hline
\end{tabular}
\end{center}

Here, ``for $\Frr$ large enough'' means that there is a finite extension $\F_{r_0}$ of $\Fp$ such 
that the statement holds for all finite extensions $\Frr$ of $\F_{r_0}$ (see Proposition~\ref{prop:ord-L-p=-1(6)}(3) for an explicit definition of $r_0$).

\subsection{Acknowledgement}
RG is supported by the Swiss National Science Foundation (SNSF Professorship \#170565 awarded to Pierre Le~Boudec), and received additional funding from ANR grant ANR-17-CE40-0012 (FLAIR). 
DU is partially supported by grant 359573 from the Simons Foundation.
Both authors thank the anonymous referee for a very careful reading of the paper and several valuable suggestions.

\section{First results}\label{s:first-results}

\subsection{Definitions and notation}\label{ss:definitions}
Notation from this section will be in force throughout the paper.  We
refer to \cite{Ulmer11} for a review of what is known about elliptic
curves over function fields, in particular with regard to the
conjecture of Birch and Swinnerton-Dyer.

Let $p>3$ be a prime number, let $\Fp$ be the field of $p$ elements, 
and fix an algebraic closure $\Fpbar$ of $\Fp$.  
Let $\Frr\subset\Fpbar$  be the finite extension of $\Fp$ of cardinality $r=p^\nu$,
and let $\K=\Frr(t)$ be the rational function field over $\Frr$.  We
write $v$ for a place of $K$, $K_v$ for the completion of $K$ at $v$,
$\deg(v)$ for the degree of $v$, $\F_v$ for the residue field at $v$,
and $r_v=r^{\deg(v)}$ for the cardinality of $\F_v$.  We identify
places of $K$ with closed points of the projective line $\P^1_{\Frr}$ over $\Frr$, and we note
that finite places of $K$ are in bijection with monic irreducible polynomials in
$\Frr[t]$. 

Let $q=p^f$ be a power of $p$, and let $E$ be the elliptic curve over
$K$ defined by
\begin{equation}\label{eq:WModel}
E=E_{q,r}: \quad y^2=x^3+t^q-t.    
\end{equation}
Write $E(K)$ for the group of $K$-rational points on $E$.  By the
Lang--N\'eron theorem, this is a finitely generated abelian group.

Let $\EE\to\P^1_{\Frr}$ be the N\'eron model of $E$.  We write $c_v$
for the number of connected components in the special fiber of $\EE$
over $v$.  One also calls $c_v$ the local Tamagawa number of $E$ at
$v$.

We denote the (differential) height  of $E$, as defined in \cite{Ulmer11}*{Lecture~3, \S2}, by $\deg(\omega_E)$.
It follows from \cite{Ulmer11}*{Lecture~3, Exer.~2.2} that for $E$, 
$$\deg(\omega_E)=\lceil q/6\rceil=\begin{cases}
\frac{q+5}6&\text{if $q\equiv1\pmod6$}\\
\frac{q+1}6&\text{if $q\equiv-1\pmod6.$}
\end{cases}$$

\subsection{Reduction types}\label{ss:reduction}
From the Weierstrass equation \eqref{eq:WModel}, one easily computes
 $$\Delta = - 2^43^3 \left(t^q-t\right)^2 \qquad \text{ and } \qquad 
 j(E) = 0.$$
Applying Tate's algorithm (see \cite{SilvermanAT}*{Chap.~IV, \S9}),
one obtains the following further facts:
\begin{itemize}
\item The curve $E$ has additive reduction of type $\II$ at
  all finite places $v$ dividing $t^q-t$, 
\item At $t=\infty$, the curve $E$ has additive reduction of type
  $\II^\ast$ if $q\equiv 1\bmod{6}$ and of type $\II$
  if $q\equiv 5\bmod{6}$.
\item The curve $E$ has good reduction at all other places of $K$.
\end{itemize}
From this collection of local information, one deduces that 
the conductor $\mathcal{N}_E$ of $E$ has degree  $\deg\mathcal{N}_E=2(q+1)$. 
One can also recover the fact that $\deg(\omega_E)=\lceil q/6\rceil$ from this computation.

\subsection{Isotriviality}
Consider the finite extension $L=K[u]/(u^6=t^q-t)$ of $K$, and let
$E_0$ be the elliptic curve over $\Frr$ defined by
$$E_0:\quad w^2=z^3+1.$$  
Then $E\times_KL$ 
is isomorphic to the constant curve
$E_0\times_{\Frr}L$ via the substitution $(x,y)=(u^2z,u^3w)$.
In other words,  $E$ is the sextic twist of $E_0$ (or rather of $E_0\times_\Frr K$) by $t^q-t$.

We record two consequences for later use.  Recall that the local
Tamagawa number $c_v$ is the number of components in the special fiber
of the N\'eron model at $v$.  Its values in terms of the local
reduction type are tabulated in \cite[p.~365]{SilvermanAT}.

\begin{prop}\label{prop:Tawagawa-torsion}
\mbox{}
\begin{enumerate}
\item For every place $v$ of $K$, the local Tamagawa number $c_v$ is $1$.
\item One has $E(K)_{\mathrm{tors}}=0$.
\end{enumerate}  
\end{prop}

\begin{proof}
  Part (1) is immediate from the table cited above.  For part (2),
  suppose that $P\in E(K)$ is a non-trivial torsion point.  Let
  $Q=(\alpha,\beta)\in E_0(L)$ be the image of $P$ under the
  above isomorphism $E\times_KL\cong E_0\times_{\Frr}L$.  Then $Q$ is
  again a torsion point, and it is known (e.g.,
  \cite[Prop.~I.6.1]{Ulmer11}) that torsion points on a constant
  curve have constant coordinates.  I.e., we have
  $\alpha,\beta\in\Frr$.  The original point $P$ thus has coordinates
  $(\alpha u^2,\beta u^3)$.  However, if $\alpha\in\Frr$, then
  $\alpha u^2\in K$ only if $\alpha=0$, and if $\beta\in\Frr$, then
  $\beta u^3\in K$ only if $\beta=0$.  Since $(0,0)\not\in E(K)$,
  there is no non-trivial torsion point $P\in E(K)$.
\end{proof}

\section{Preliminaries on exponential sums}\label{s:exp-sums}

\subsection{Finite fields}\label{ss:finite-fields}
Fix an algebraic closure $\Qbar$ of $\Q$ and a prime ideal $\gP$ above
$p$ in the ring of algebraic integers $\Zbar\subset \Qbar$.  The
quotient $\Zbar/\gP$ is then an algebraic closure of $\Fp$ which we
denote by $\Fpbar$.   All finite fields in this paper will be viewed as
subfields of this $\Fpbar$.

\subsection{Multiplicative characters}\label{ss:mult-chars}
Reduction modulo $\gP$ induces an isomorphism between the roots of
unity of order prime to $p$ in $\Zbar$ and $\Fpbar^\times$.  We let
$\teich:\Fpbar^\times\to\Qbar^\times$ denote the inverse of this
isomorphism.  The same letter $\teich$ will be used to denote the
restriction of $\teich$ to the multiplicative group of any finite
extension $\F$ of $\Fp$ ($\F$ being viewed as a subextension of
$\Fpbar$).

If $\F$ is a finite extension of $\Fp$ and $n$ is a divisor of
$|\F^\times|$, define
$$\chi_{\F,n}:=\teich^{|\F^\times|/n}.$$
This is a character of $\F^\times$ of order exactly $n$.  In
particular, if $n=|\F^\times|$, the character $\chi_{\F, n}$ is a
generator of the group of multiplicative characters of $\F$.

If $\F\subset\F'$ are finite extensions of $\Fp$, if $n$ divides the
order of $\F^\times$,  and if $\N_{\F'/\F}$
denotes the norm from $\F'$ to $\F$, then an elementary calculation
shows that 
$\chi_{\F',n}=\chi_{\F,n}\compose\N_{\F'/\F}.$

\subsection{Additive characters}\label{ss:add-chars}
Fix once and for all a non-trivial additive character
$$\psi_p:\Fp\to\Q(\mu_p)^\times\subset\Qbar^\times.$$  If $\F$ is a finite
extension of $\Fp$, if $\Tr_{\F/\Fp}$ denotes the trace from $\F$ to
$\Fp$, and if $\alpha\in\F^\times$, then the map $x\mapsto\psi_\alpha(x)$
defined by
$$\psi_\alpha(x)=\psi_p\left(\Tr_{\F/\Fp}(\alpha x)\right)$$
for all $x\in\F$ is a non-trivial additive character of $\F$. 
Moreover, any non-trivial
additive character of $\F$ is of the form $\psi_\alpha$ for a unique
$\alpha\in\F^\times$.  When we need to make the underlying field precise, we
write $\psi_{\F,\alpha}$ instead of $\psi_\alpha$.

\subsection{Gauss sums}\label{ss:GaussSums}
If $\F$ is a finite extension of $\Fp$, $\chi$ is a non-trivial
character of $\F^\times$, and $\psi$ is a non-trivial additive
character of $\F$, define the Gauss sum $G_\F(\chi,\psi)$ by
$$G_\F(\chi,\psi)=-\sum_{x\in\F^\times}\chi(x)\psi(x).$$

We recall a few well-known properties of these Gauss sums:
\begin{enumerate}
\item\label{item.Gauss.integer} If $\chi$ has order $n$, the sum $G_\F(\chi,\psi)$ is an
  algebraic integer in $\Q(\mu_{np})$.
\item\label{item.Gauss.magnitude}
For any non-trivial characters $\chi$ and $\psi$,   one has $|G_\F(\chi,\psi)| = |\F|^{1/2}$ in any complex embedding of $\Qbar$.
\item\label{item.Gauss.change.add.char} 
For all non-trivial  multiplicative characters $\chi$ 
  on $\F^\times$ and all $\alpha\in\F^\times$, one has
$$G_\F(\chi,\psi_\alpha) = \chi^{-1}(\alpha)G_\F(\chi,\psi_1).$$ 
\item\label{item.Gauss.HD} (Hasse--Davenport relation) Let $\chi$ be a
  non-trivial multiplicative character on $\F^\times$ and $\psi$ be a
  non-trivial additive character on $\F$.  Then for any finite
  extension $\F'/\F$, one has
$$G_{\F'}(\chi\compose \N_{\F'/\F},\psi\compose \Tr_{\F'/\F})
=G_\F(\chi,\psi)^{[\F':\F]}.$$
\item (Stickelberger's Theorem)  Let $\ord$ be the $p$-adic valuation
  of $\Qbar$ associated to $\gP$, normalized so that $\ord(p)=1$.  If
  $\F$ has cardinality $p^\mu$ and $0<s<p^\mu-1$ has $p$-adic
  expansion
$$s=s_0+s_1p+\cdots+s_{\mu-1}p^{\mu-1}$$
with $0\le s_i<p$, then
$$\ord(G_\F(\chi_{\F,|\F^\times|}^{-s},\psi))=\frac1{p-1}\sum_{i=0}^{\mu-1}s_i.$$
\end{enumerate}
These results are classical, and the reader may find proofs of them
(and the claims in the next two subsections) in \cite{WashingtonCF}*{Chap.~VI, \S1-\S2} for instance.

\subsection{Explicit Gauss sums}\label{ss:explicit-sums}
Let $\F$ be a finite extension of $\F_p$, and write $|\F|=p^\mu$.
An elementary calculation shows that, for any non-trivial additive character $\psi$ of $\F$, one has
\begin{equation}\label{eq:quad-Gauss-sum}
G_\F(\chi_{\F,2},\psi)^2=((-1)^{(p-1)/2} p)^{\mu}.
\end{equation}
In particular,  $\ord G_\F(\chi_{\F, 2}, \psi)= \mu/2$. Here, as above, $\ord$ denotes the $p$-adic valuation on $\Qbar$ associated to $\gP$, normalized to that $\ord(p)=1$.

If $p\equiv1\pmod3$, 
then Stickelberger's theorem (see (5) above) shows that for any non-trivial additive character
$\psi$ of $\F$, one has 
\begin{equation}\label{eq:cubic-Gauss-sum1}
\ord G_\F(\chi_{\F,3},\psi)=\frac23\mu\quad\text{and}\quad
\ord G_\F(\chi_{\F,3}^{-1},\psi)=\frac13\mu.  
\end{equation}
On the other hand, if $p\equiv2\pmod3$, then $3$ divides $|\F^\times|$ if and only if $\mu=[\F:\F_p]$ is even. 
If this is the case (i.e., if $|\F|=p^\mu\equiv1\pmod3$),
an old result of Tate and Shafarevich (see
\cite[Lemma~8.2]{Ulmer02}) and the Hasse--Davenport relation yield that
$$G_\F(\chi_{\F,3},\psi_1)=G_\F(\chi_{\F,3}^{-1},\psi_1)=(-p)^{\mu/2},$$
and therefore (see (3) in the previous subsection)
\begin{equation}\label{eq:cubic-Gauss-sum2}
G_\F(\chi_{\F,3},\psi_\alpha)=\chi_{\F,3}^{-1}(\alpha)(-p)^{\mu/2}
\quad\text{and}\quad 
G_\F(\chi_{\F,3}^{-1},\psi_\alpha)=\chi_{\F,3}(\alpha)(-p)^{\mu/2}.
\end{equation}
In particular, $\ord G_\F(\chi_{\F,3}^{\pm1},\psi_\alpha)= \mu/2$ in this case.

\subsection{Jacobi sums}\label{ss:JacobiSums}
We require only the simplest case:  Let $\F$ be a finite extension of
$\Fp$ and let $\chi_1$ and $\chi_2$ be two non-trivial characters of
$\F^\times$ such that $\chi_1\chi_2$ is also non-trivial.  Define
$$J_\F(\chi_1,\chi_2)=-\sum_{x\in\F}\chi_1(x)\chi_2(1-x).$$

An elementary calculation (again, see \cite[Chap.~VI]{WashingtonCF})
shows that
\begin{equation}\label{eq:GaussJacobi}
J_\F(\chi_1,\chi_2)=\frac{G_\F(\chi_1,\psi)G_\F(\chi_2,\psi)}
{G_\F(\chi_1\chi_2,\psi)}  
\end{equation}
for any non-trivial additive character $\psi$ of $\F$.  One may then
deduce the archimedean and $p$-adic sizes of $J(\chi_1,\chi_2)$
from the results quoted in Section~\ref{ss:GaussSums}.

\subsection{Orbits}\label{ss:orbits}
Recall that $p>3$ is a prime. 
Given an integer $n\geq 1$ prime to $p$, let 
$$S=S_{n,q}=\left(\Z/n\Z\setminus\{0\}\right)\times\Fqtimes
\quad\text{and}\quad
S^\times=S_{n,q}^\times=(\Z/n\Z)^\times\times\Fqtimes.$$ 
Let $r=p^\nu$ for some positive integer $\nu$. Write $\<r\>$ for the subgroup of $\Q^\times$ generated by
$r$, and consider the action of $\<r\>$  on $S$ and $S^\times$ given by the rule
$$\forall (i, \alpha)\in S, \qquad r(i,\alpha):=(ri,\alpha^{1/r}).$$
In other words, $r$ acts on $\Z/n\Z$ by multiplication, and on $\Fqtimes$ by the inverse of the $r$-power Frobenius. 
Let $O_{r,n,q}$ be the set of orbits of $\<r\>$ on $S$ and
$O_{r,n,q}^\times$ the set of orbits on $S^\times$.

If $n=1$, then $O_{r,n,q}^\times$ is just the set of orbits of $\<r\>$
on $\Fqtimes$, which we denote by $O_{r,q}$.  Note that if $o\in O_{r,q}$
is the orbit through $\alpha$, then the cardinality $|o|$ of $o$  is
equal to the degree $[\Frr(\alpha):\Frr]$ of the field extension
$\Frr(\alpha)$ over $\Frr$.

For a general $n$, if $o\in O_{r,n,q}^\times$ is the orbit through
$(i,\alpha)$, then
\begin{equation}\label{eq:orbit-size}
|o|=\lcm\left(\ord^\times(r\bmod n),[\Frr(\alpha):\Frr]\right)
\end{equation}
where $\ord^\times(r\bmod n)$ denotes the order of $r$ in
$(\Z/n\Z)^\times$. 
Note that, for any $\alpha\in\F_q$, one has $[\F_r(\alpha):\F_r] = \lcm(\nu, [\F_p(\alpha):\F_p])/[\F_p(\alpha):\F_p]$, and $[\F_p(\alpha):\F_p]$ divides $f=[\F_q:\F_p]$. 
It is  then clear that $|o|$ divides $\lcm\left(\ord^\times(r\bmod n),\lcm(f,\nu)/f\right)$ for any orbit $o\in O^\times_{r,n,q}$.

In what follows, we will only need the cases where $n$ divides 6.  
If $r\equiv1\pmod6$, then   $\<r\>$ acts trivially on $\Z/6\Z$ and  the
orbits $o\in O_{r,6,q}$ are ``vertical'' in the sense that they are of
the form $o=\{(i,\alpha)\}$ where $i$ is fixed and $\alpha$ runs
through an orbit of $\<r\>$ on $\Fqtimes$.  In particular,
$|o|=[\Frr(\alpha):\Frr]$.

On the other hand, if $r\equiv5\equiv-1\pmod6$, then orbits $o\in
O_{r,6,q}$ ``bounce left and right'' in the sense that an orbit $o$ contains
elements $(i,\alpha)$ and $r(i,\alpha)=(-i,\alpha^{1/r})$.  In this case, if
$o$ is the orbit through $(i,\alpha)$, then 
$|o|=\lcm(2,[\Frr(\alpha):\Frr])$.  

In both cases (that is to say, for $r\equiv\pm1 \pmod 6$), note that $\nu |o|$ is even for all orbits $o\in O^\times_{r, 6, q}$.

For $n\in\{2,3\}$, the natural projection $(\Z/6\Z)^\times\to(\Z/n\Z)^\times$ induces a map $\pi_n:O_{r,6,q}^\times\to O_{r,n,q}^\times$. 
We record a few elementary observations about $\pi_n$:
\begin{itemize}
    \item The map $\pi_3$
is a bijection, because $(\Z/6\Z)^\times\to(\Z/3\Z)^\times$ is a
bijection.
    \item If $r\equiv1\pmod6$, then $\pi_2$ is two-to-one.  (This is
essentially the same point as the ``vertical'' remark above.)
    \item    If
$r\equiv-1\pmod6$  and if $o'\in O_{r,2,q}^\times$ has $|o'|$ even, then there
are two orbits $o\in O_{r,6,q}^\times$ with $\pi_2(o)=o'$.  
Finally, if $r\equiv-1\pmod6$  and if $o'\in O_{r,2,q}^\times$ has $|o'|$
odd, then there is a unique orbit $o\in O_{r,6,q}^\times$ with $\pi_2(o)=o'$
and the underlying map of sets $o\to o'$ is two-to-one.
\end{itemize}
Motivated by this last remark, for any $o\in O_{r,6,q}^\times$, we define
$$m_2(o)=\frac{|o|}{|\pi_2(o)|}.$$
Thus $m_2(o)=1$ unless $r\equiv-1\pmod6$ and $|\pi_2(o)|$ is odd, in
which case $m_2(o)=2$.

\subsection{Gauss sums associated to orbits}\label{ss:G(o)}
Fix data $p$, $r$, $q$, and $n$ as above, and let $o\in O_{r,n,q}$ be
the orbit of $\<r\>$ through
$(i,\alpha)\in S_{n,q}=\left(\Z/n\Z\setminus\{0\}\right)\times\Fqtimes$.
Let $\F=\F_{r^{|o|}}$, i.e., $\F$ is the extension of~$\Frr$ of degree
$|o|$.  By formula~\eqref{eq:orbit-size} for $|o|$, $\F$ can be interpreted as the
smallest extension of~$\Frr$ which admits a multiplicative character
of order $n$ and contains $\alpha$. 
To the orbit $o$ we  then associate the Gauss sum
\begin{equation}\label{eq:G(o)-def}
G(o)=G_\F(\chi_{\F,n}^i,\psi_\alpha),  
\end{equation}
where $\chi_{\F, n}$ and $\psi_\alpha$ are the characters on $\F$ defined in Sections~\ref{ss:mult-chars} and \ref{ss:add-chars}. 
An elementary computation, as in \cite[Lemma~2.5.8]{CohenNT1}, shows
that
$G_\F(\chi,\psi_{\alpha})=G_\F(\chi^p,\psi_{\alpha^{1/p}}),$
so that $G(o)$ is indeed well defined independently of the choice of element $(i,\alpha)\in o$.

We next record the valuations of Gauss sums associated to orbits for
$n=2$ and $3$.  These claims follow immediately from the results of
Section~\ref{ss:explicit-sums}.  

When $n=2$, we have $\ord(G(o))=\nu|o|/2$ for all orbits $o\in O_{r,2,q}^\times$.

When $n=3$, $p\equiv1\pmod3$,  and $o\in O_{r,3,q}^\times$, then 
$$\ord(G(o))=\begin{cases}
\frac23\nu|o|&\text{if $o$ contains an element $(1,\alpha)$}\\
\frac13{\nu|o|}&\text{if $o$ contains an element $(-1,\alpha)$.}
\end{cases}$$

When $n=3$ and $p\equiv-1\pmod3$, then $\ord(G(o))=\frac12{\nu|o|}$
for all $o\in O_{r,3,q}^\times$.

The following shows that the Gauss sums $G(o)$ ``decompose'' as roots of unity
times powers of Gauss sums of small weight.  
This will play a key role in our estimation of the archimedean size of $\Reg(E)|\sha(E)|$ in  Section~\ref{s:BS}.

\begin{prop}\label{prop:G-power}
Let $n\geq 1$ be an integer coprime to $p$, and write $c:= \ord^\times(p\bmod{n})$ for the order of $p$ modulo $n$. 
Then for all $o\in O_{r,n,q}$, one has
$$G(o)=\zeta g^{|o|\nu/c}$$
where $\zeta$ is a $n$-th root of unity, and $g\in\Q(\mu_{np})$ a Weil integer of size $p^{c/2}$.
\end{prop}

Recall that an algebraic number $z\in\Qbar$ is called \emph{a Weil integer of size $p^a$} (with
$a\in\frac12\Z_{\ge0}$) if $z$ is an algebraic integer such that
$|z|=p^a$ in any complex embedding $\Q(z)\hookrightarrow\C$.  (These
numbers are also sometimes called $p$-Weil integers of weight $2a$.)


\begin{proof}
  Note that $\F_{p^c}^\times$ admits characters of order exactly $n$.  
  By definition, for any choice of representative $(i, \alpha)\in o$, we have
$$G(o)=G_\F(\chi_{\F,n}^i,\psi_{\F,\alpha})$$
where $\F$ is the extension of $\Frr$ of degree $|o|$, i.e.,
$|\F|=p^{|o|\nu}$. By construction,  $c$ divides $\nu|o|$, so $\F$ is an extension of $\F_{p^c}$.
Then the following holds: 
\begin{align*}
G(o)&=G_\F(\chi_{\F,n}^i,\psi_{\F,\alpha})
=\chi_{\F,n}^{-i}(\alpha)G_\F(\chi_{\F,n}^i,\psi_{\F,1})
&\text{(by (3) in Section~\ref{ss:GaussSums})}\\
&=\chi_{\F,n}^{-i}(\alpha)G_{\F_{p^c}}(\chi_{\F_{p^c},n}^i,\psi_{\F_{p^c},1})^{|o|\nu/c}
&\text{(by the Hasse--Davenport relation)}.
\end{align*}
We now let $\zeta:=\chi_{\F,n}^{-i}(\alpha)$ and $g=G_{\F_{p^c}}(\chi_{\F_{p^c},n},\psi_{\F_{p^c},1})$.
Since $\chi_{\F,n}$ has order $n$, $\zeta$ is a $n$-th root of unity. 
By \eqref{item.Gauss.integer} and \eqref{item.Gauss.magnitude} in Section~\ref{ss:GaussSums}, $g$ is a Weil integer in $\Q(\mu_{np})$ of size $p^{c/2}$.
\end{proof}

\subsection{Jacobi sums associated to orbits}\label{ss:J(o)}
With data $p$ and $r$ as usual, let $\<r\>$ act on $(\Z/6\Z)^\times$ by
multiplication, and let $N=N_{r,6}$ be the set of orbits of $\<r\>$ on
$(\Z/6\Z)^\times$.  Thus, if $r\equiv1\pmod6$, there are two orbits,
both singletons, and if $r\equiv-1\pmod6$, there is a unique orbit,
$o=\{1,-1\}$.  (This is a somewhat trivial situation, but we introduce
it for consistency with our treatment of Gauss sums.)
Given $o\in N_{r,6}$, write $\F=\F_{r^{|o|}}$ and associate to $o$ the
Jacobi sum
\begin{equation}\label{eq:def-J(o)}
J(o):=J_\F(\chi_{\F,2}^{-i},\chi_{\F,3}^{-i})
=J_{\F}(\chi_{\F,6}^{-3i},\chi_{\F,6}^{-2i})
\end{equation}
for any $i\in o$.  As a straightforward calculation shows, one has $J_\F(\chi_1^p,\chi_2^p)=J_\F(\chi_1,\chi_2)$, so that
the sum $J(o)$ is well defined independently of the choice of $i\in o$.

We next record the valuations of $J(o)$ for $o\in N_{r,6}$.  These
claims follow easily from the expression of Jacobi sums in terms of
Gauss sums and Stickelberger's theorem (see Sections~\ref{ss:GaussSums} and~\ref{ss:JacobiSums}).  If $p\equiv-1\pmod6$, then
$$\ord(J(o))=\frac12{\nu|o|}$$
for all $o\in N_{r,6}$.  On the other hand, if $p\equiv1\pmod6$, then
$$\ord(J(\{1\}))=0\quad\text{and}\quad\ord(J(\{-1\}))=\nu.$$

Finally, we introduce the map $\rho_6:O_{r,6,q}^\times\to N_{r,6}$
induced by the projection 
$$(\Z/6\Z)^\times\times\Fq^\times\to(\Z/6\Z)^\times.$$
This will play a role in our geometric calculation of the $L$-function $L(E,s)$ in Section~\ref{s:L-cohom}.

\section{Elementary calculation of the $L$-function}\label{s:L-elem}
Recall that we have fixed a prime number $p>3$, a finite field $\Frr$ of
characteristic $p$, a power $q$ of $p$, and that we have defined $E=E_{q,r}$ as the
elliptic curve
$$E:\quad y^2=x^3+t^q-t$$
over $K=\Frr(t)$.  In this section, we give an elementary calculation
of the Hasse--Weil $L$-function of $E$ over $K$. 
The Hasse--Weil $L$-function of $E$ is defined as the Euler product
$$L(E,T)=\prod_{\text{good $v$}}
\left(1-a_vT^{\deg(v)}+r_vT^{2\deg(v)}\right)^{-1}
\prod_{\text{bad $v$}}\left(1-a_vT^{\deg(v)}\right)^{-1},$$ 
where the products are over places $v$ of $K$.
Here ``good $v$'' refers to the places where $E$ has good reduction,
``bad $v$'' refers to the places of bad reduction, and for any place $v$,
$\F_v$ is the residue field at $v$, $r_v$ is its cardinality, and
$a_v$ is the integer such that the number of points on the
plane cubic model of $E$ over $\F_v$ is equal to $r_v-a_v+1$.  Note that, since
$E$ has additive reduction at all bad places
(Section~\ref{ss:reduction}), the local factors at such places are all
$1$, so
\begin{equation}\label{eq:L-elem}
L(E,T)=\prod_{\text{good $v$}}\left(1-a_vT^{\deg(v)}+r_vT^{2\deg(v)}\right)^{-1}.
\end{equation}

One also considers $L(E,s)=L(E,T)$ with $T=r^{-s}$.  Since the curve $E$ is non-constant, it is known 
(e.g., \cite[Lecture~1, Thm.~9.3]{Ulmer11}) that
$L(E,s)$ is a polynomial in $T=r^{-s}$ and that it satisfies a
functional equation relating $L(E,s)$ and $L(E,2-s)$.

Recall from Section~\ref{ss:orbits} that $O_{r,n,q}^\times$ denotes the 
set of orbits of $\<r\>$ acting on $(\Z/n\Z)^\times\times\Fqtimes$, that 
$\pi_n:O_{r,6,q}^\times\to O_{r,n,q}^\times$ (for $n=2,3$) denotes the
map induced by the natural projection
$(\Z/6\Z)^\times\to(\Z/n\Z)^\times$, and that
$m_2(o)=\frac{|o|}{|\pi_2(o)|}$.  
As in Section~\ref{ss:G(o)}, we attach a Gauss sum  $G(o)$  to any orbit  $o\in O^\times_{r,n,q}$.

The main result of this section is the following.

\begin{thm}\label{thm:L-elem} In the above setting, we have
$$L(E,s)=\prod_{o\in O_{r,6,q}^\times}
\left(1-G(\pi_2(o))^{m_2(o)}G(\pi_3(o))r^{-s|o|}\right).$$ 
\end{thm}

Note that, as a polynomial in $r^{-s}$, the $L$-function has degree $\sum_{o\in O_{r, 6,q}^\times} |o| = |S_{6,r,q}^\times| = 2(q-1)$. This is consistent with what the Grothendiek--Ogg--Shafarevich formula predicts, namely that the $L$-function has degree $\deg(\NN_E)-4$ where $\NN_E$ is the conductor of $E$ (recall from Section~\ref{ss:reduction} that $\deg\NN_E = 2(q+1)$).

The first, elementary, proof of Theorem~\ref{thm:L-elem} will be given at the end of this section, after proving several lemmas in the next few subsections. 
In Section~\ref{s:L-cohom}, we will provide two more conceptual proofs of this statement (see Theorems~\ref{thm:ST-L} and \ref{thm:AS-L}, as well as Section~\ref{ss:L-compare}).

\begin{lemmas}\label{lemma:S}
Let $\F$ be a finite field of characteristic $p$, and let $\psi$ be a non-trivial additive character of~$\F$.
\begin{enumerate}
\item 
For any $u\in\F$ and any power $q$ of $p$, one has
$$\left|\{t\in\F : t^q-t=u\}\right|=
\sum_{\alpha\in\F\cap\Fq}\psi(\alpha u).$$
\item   
Denote   the non-trivial quadratic character of $\F^\times$ by $\lambda=\chi_{\F,2}$. Consider the sum
\begin{equation}\label{eq:def-S-sum}
    S_\F(\lambda,\psi)=\sum_{x,z\in\F}\lambda(x^3+z)\psi(z).
\end{equation}
Then
$$S_\F(\lambda,\psi)=\begin{cases}
0&\text{if $|\F|\equiv 2\pmod3$}\\
G_\F(\lambda,\psi)\sum_{i\in\{1,2\}}G_\F(\chi_{\F,3}^i,\psi)&
\text{if $|\F|=1\pmod3$}.
\end{cases}$$
\end{enumerate}
\end{lemmas}

\begin{proof}
Part (1) is straightforward when $\F$ is an extension of $\Fq$, and the general case 
is proven in
  \cite[Lemma~4.3]{Griffonpp1801}.  (The key point is that the kernel
  and the image of the map $\F\to\F$, $t\mapsto t^q-t$ are orthogonal
  complements with respect to the $\Fp$-bilinear form
  $\<\alpha,\beta\>=\Tr_{\F/\Fp}(\alpha\beta)$.) 
  We now turn to the proof of (2). For any non-trivial additive character $\psi$ on $\F$,  consider
$$S_\F(\lambda,\psi)=\sum_{x,z\in\F}\lambda(x^3+z)\psi(z).$$
Let $\trivcar$ denote the trivial multiplicative character of $\F^\times$.
It is classical that for any $y\in\F$,
$$\big|\big\{ x\in\F :  y=x^3\big\}\big|=\sum_{\theta^3=\trivcar} \theta(y)$$
where the sum runs over characters on $\F^\times$ whose order divides
$3$ (see \cite[Lemma 2.5.21]{CohenNT1}). This allows us to rewrite the
sum $S_\F(\lambda,\psi)$ as 
\begin{align*}
S_\F(\lambda,\psi)
&= \sum_{y\in\F}\sum_{z\in\F} 
  \left(\sum_{\theta^3=\trivcar}\theta(y)\right) \lambda(y+z) \psi(z)\\
&= \sum_{\theta^3=\trivcar}\sum_{y\in\F} \theta(y)
  \left(\sum_{z\in\F} \lambda(y+z) \psi(z)\right) \\
&= \sum_{\theta^3=\trivcar}\sum_{y\in\F} \theta(y)
   \left(\sum_{u\in\F} \lambda(u) \psi(u-y)\right) 
    &{(\text{by setting }u=z+y)}\\
&= \left(\sum_{\theta^3=\trivcar}\sum_{y\in\F} \theta(y)\psi(-y)\right)
   \left(\sum_{u\in\F} \lambda(u) \psi(u)\right)\\
&= \left(\sum_{u\in\F} \lambda(u) \psi(u)\right)
   \left(\sum_{\theta^3=\trivcar}\theta(-1)\sum_{v\in\F}\theta(v)\psi(v)\right) 
&{(\text{by setting } v=-y)}.
\end{align*}
The first sum equals $-G_\F(\lambda,\psi)$ and, for a character
$\theta$ such that $\theta^3=\trivcar$, the sum over $v\in\F$ equals
$-G_\F(\theta,\psi)$.  Moreover, $\theta(-1)=1$ for all $\theta$ such
that $\theta^3=\trivcar$, and $G_\F(\trivcar,\psi)=0$, so we have
$$S_\F(\lambda,\psi)=G_\F(\lambda,\psi)
\sum_{\substack{\theta^3=\trivcar\\\theta\neq\trivcar}}G_\F(\theta,\psi).$$
To conclude the proof, it remains to note that if $|\F|\equiv2\pmod3$,
then there are no non-trivial characters of order 3, so the right hand
side vanishes, while if $|\F|\equiv1\pmod3$, the two non-trivial
characters of order 3 are $\chi_{\F,3}^i$, $i\in\{1,2\}$.
\end{proof}

To ease notation, for the rest of this section we write $\F_n$ for
$\F_{r^n}$, i.e., $\F_n$ is the extension of $\Frr$ of degree $n$.
Fix a non-trivial additive character $\psi_{\F_n}$ of $\F_n$ and for any $\alpha\in\F_n$, let $\psi_{\F_n,\alpha}$ denote the additive character on $\F_n$ defined by $z\in\F_n \mapsto\psi_{\F_n}(\alpha z)$. 

\begin{lemmas}\label{lemma:L-lhs} 
As Taylor series in $T$,
$$-\log L(E,T)=
\sum_{n\ge1}\frac{T^n}n\sum_{\alpha\in\F_n\cap\Fq}
S_{\F_n}(\lambda_{\F_n},\psi_{\F_n,\alpha})$$
where $\lambda_{\F_n}=\chi_{\F_n,2}$ is the non-trivial quadratic
character of $\F_n^\times$ and $S_{\F_n}(\lambda_{\F_n}, \psi_{\F_n, \alpha})$ is the sum defined by equation \eqref{eq:def-S-sum}. 
\end{lemmas}

\begin{proof}
In the definition of $L(E, T)$, write the Euler factor at a good place $v$ as
$$\left(1-a_vT^{\deg(v)}+r_vT^{2\deg(v)}\right)
=\left(1-\alpha_vT^{\deg(v)}\right)\left(1-\beta_vT^{\deg(v)}\right).$$
Taking the logarithm of the Euler product \eqref{eq:L-elem} and reordering terms yields that
$$\log L(E,T)=\sum_{n\ge1}\frac{T^n}n
\sum_{\substack{\text{good $v$}\\\deg(v)|n}}
\deg(v)\left(\alpha_v^{n/\deg(v)}+\beta_v^{n/\deg(v)}\right).$$
To obtain this expression, we have used the standard identity between Taylor series:
\begin{equation}\label{eq:log-taylor-exp}
\log(1-\alpha T)=-\sum_{n\ge1}\frac{(\alpha T)^n}{n}.    
\end{equation}
If $t\in\F_n$, define $A_E(t,n)$ to be the integer such that $r^n+1-A_E(t,n)$ is the number of
$\F_n$-rational points on the reduction of $E$ at $t$.  It then follows from
\cite[V.2.3.1]{SilvermanAEC} that
$$ \alpha_v^{n/\deg(v)}+\beta_v^{n/\deg(v)} =
A_E(t,n)$$
for any $t\in\F_n$ lying over $v$.  Thus,
$$L(E,T)=\sum_{n\ge1}\frac{T^n}n
\sum_{\substack{\text{good $t$}\\t\in\F_n}}
A_E(t,n).$$
Denote  the non-trivial quadratic character of
$\F_n^\times$ by $\lambda_{\F_n}$. Then \cite[V.1.3]{SilvermanAEC} asserts that
$$A_E(t,n)=-\sum_{x\in\F_n}\lambda_{\F_n}(x^3+t^q-t).$$
Note that if $t\in\Fq$, then $t^q-t=0$, and the sum on the right hand side
vanishes, so we may drop the restriction ``good $t$'' in the last
expression for $L(E,T)$, i.e.,
$$-\log L(E,T)=\sum_{n\ge1}\frac{T^n}n
\sum_{t\in\F_n}\sum_{x\in\F_n}\lambda_{\F_n}(x^3+t^q-t).$$

Now applying Lemma~\ref{lemma:S} part (1), we get that 
\begin{align*}
\sum_{t\in\F_n}\sum_{x\in\F_n}\lambda_{\F_n}(x^3+t^q-t)
&=\sum_{x\in\F_n}\sum_{u\in\F_n}
\sum_{\alpha\in\F_n\cap\Fq}\psi(\alpha u)\lambda_{\F_n}(x^3+u)\\
&=\sum_{\alpha\in\F_n\cap\Fq}S_{\F_n}(\lambda_{\F_n},\psi_{\F_n,\alpha}).
\end{align*}
Therefore, we have proved, as desired, that
$$-\log L(E,T)=\sum_{n\ge1}\frac{T^n}n\sum_{\alpha\in\F_n\cap\Fq}
S_{\F_n}(\lambda_{\F_n},\psi_{\F_n,\alpha}).$$
\end{proof}

\begin{lemmas}\label{lemma:L-rhs} 
As Taylor series in $T$,
\begin{multline*}
-\log \prod_{o\in O_{r,6,q}^\times}
\left(1-G(\pi_2(o))^{m_2(o)}G(\pi_3(o))T^{|o|}\right)\\
=\sum_{\substack{n\ge1\\r^n\equiv1\,(\mathrm{mod}\,6)}}\frac{T^n}n
\sum_{\alpha\in\F_n\cap\Fq}\sum_{i\in\{1,2\}}
G_{\F_n}(\chi_{\F_n,2},\psi_{\F_n,\alpha})
G_{\F_n}(\chi_{\F_n,3}^i,\psi_{\F_n,\alpha}).
\end{multline*}
\end{lemmas}

\begin{proof} 
To lighten the notation, we write $\omega(o) := G(\pi_2(o))^{m_2(o)}G(\pi_3(o))$ for any $o\in O_{r,6,q}^\times$.
By identity \eqref{eq:log-taylor-exp}, we have 
\begin{equation*}
-\log \prod_{o\in O_{r,6,q}^\times}
\left(1- \omega(o)T^{|o|}\right)\\
=\sum_{n\ge1}\frac{T^n}n\sum_{\substack{o\in
    O_{r,6,q}^\times\\\text{$|o|$ divides $n$}}}
|o| \omega(o)^{n/|o|}. 
\end{equation*}
Write $\F_o$ for $\F_{r^{|o|}}$, the extension of $\Frr$ of degree
$|o|$.  
Pick a representative $(i,\alpha)\in o$. By definition, we have 
$G(\pi_3(o))=G_{\F_o}(\chi_{\F_o,3}^i,\psi_{\F_o,\alpha})$
and the Hasse--Davenport relation
(Section~\ref{ss:GaussSums}) yields that
$$G(\pi_3(o))^{n/|o|}=
G_{\F_n}(\chi_{\F_n,3}^i,\psi_{\F_n,\alpha}).$$
Similarly, using the definition and the Hasse--Davenport relation, we have
$$G(\pi_2(o))^{m_2(o)n/|o|}=G_{\F_n}(\chi_{\F_n,2},\psi_{\F_n,\alpha}).$$
Note that $|o|$ divides $n$ if and only if $r^n\equiv1\pmod6$ and
$\alpha\in\F_n$. 
Thus, 
\begin{multline*}
-\log \prod_{o\in O_{r,6,q}^\times}
\left(1-\omega(o) T^{|o|}\right)\\
=\sum_{\substack{n\ge1\\r^n\equiv1\,(\mathrm{mod}\,6)}}\frac{T^n}n
\sum_{\alpha\in\F_n\cap\Fq}\sum_{i\in\{1,2\}}
G_{\F_n}(\chi_{\F_n,2},\psi_{\F_n,\alpha})G_{\F_n}(\chi_{\F_n,3}^i,\psi_{\F_n,\alpha}).
\end{multline*}
This completes the proof of the lemma.
\end{proof}

\begin{proof}[Proof of Theorem~\ref{thm:L-elem}]
According to Lemma~\ref{lemma:L-lhs},
$$-\log L(E,T)=
\sum_{n\ge1}\frac{T^n}n\sum_{\alpha\in\F_n\cap\Fq}
S_{\F_n}(\lambda_{\F_n},\psi_{\F_n,\alpha}),$$
and part (2) of Lemma~\ref{lemma:S}  
says that
$$S_{\F_n}(\lambda_{\F_n},\psi_{\F_n,\alpha})=\begin{cases}
0&\text{if $|\F_n|=r^n\equiv2\pmod3$}\\
\sum_{i\in\{1,2\}}
G_{\F_n}(\chi_{\F_n,2},\psi_{\F_n,\alpha})
G_{\F_n}(\chi_{\F_n,3}^i,\psi_{\F_n,\alpha})
&\text{if $|\F_n|=r^n\equiv1\pmod3$.}
\end{cases}$$
Noting that $r^n\equiv1\pmod3$ if and only if $r^n\equiv1\pmod6$,
we have
$$-\log L(E,T)
=\sum_{\substack{n\ge1\\r^n\equiv1\,(\mathrm{mod}\,6)}}\frac{T^n}n
\sum_{\alpha\in\F_n\cap\Fq}\sum_{i\in\{1,2\}}
G_{\F_n}(\chi_{\F_n,2},\psi_{\F_n,\alpha})
G_{\F_n}(\chi_{\F_n,3}^i,\psi_{\F_n,\alpha}).
$$
By Lemma~\ref{lemma:L-rhs}, the expression on the right hand side is 
$$-\log \prod_{o\in O_{r,6,q}^\times}
\left(1-G(\pi_2(o))^{m_2(o)}G(\pi_3(o))T^{|o|}\right),$$
thus concluding the proof of the Theorem.
\end{proof}

\section{Auxiliary curves}\label{s:curves}
In this section, we record some well-known facts about the geometry of
certain curves to be used in the sequel.

\subsection{Cohomology}
Throughout this section and the next, we denote by
$H^n(-)$ any rational Weil cohomology theory (with coefficients in an algebraically
closed field) for varieties over $\Frr$, for example
$\ell$-adic cohomology $H^n(-\times_{\Frr}\overline{\F}_r,\Qlbar)$ or
crystalline cohomology $H^n(-/W)\tensor_{W(\Frr)}\Qpbar$.  (See, for
example, \cite{Kleiman68}.)  Among other things, these groups admit a
functorial action of the geometric Frobenius $\Fr_r$.

Here is a well-known lemma about characteristic polynomials in
induced representations.  See \cite[Lemma~1.1]{Gordon79} or
\cite[Lemma~2.2]{Ulmer07b} for a proof.

\begin{lemma}\label{lemma:charpoly}
Let $V$ be a finite-dimensional vector space with subspaces $W_i$
indexed by $i\in\Z/m\Z$ such that $V=\oplus_{i\in\Z/m\Z}W_i$.  Let
$\phi:V\to V$ be a linear transformation such that $\phi(W_i)\subset
W_{i+1}$ for all $i\in\Z/m\Z$.  Then
\begin{equation*}
\det(1-\phi T|V)=\det(1-\phi^{m}T^m|W_0).
\end{equation*}
\end{lemma}

\subsection{An elliptic curve}\label{ss:E0}
We have already introduced the elliptic curve
$$E_0:\quad w^2=z^3+1$$
over $\Frr$.  The displayed equation defines a smooth affine curve, and
there is a unique point at infinity on $E_0$ which we denote by $O\in
E_0$. 

The curve $E_0$ carries an action of $\mu_6$ via
$\zeta(z,w)=(\zeta^2z,\zeta^3w)$.  The character group of $\mu_6$ is
$\Z/6\Z$.  It is well known that $H^1(E_0)$ has dimension 2, and
that under the action of $\mu_6$, it decomposes as the direct sum of
two lines corresponding to the subspaces where $\zeta\in\mu_6$ acts by
$\zeta$ and $\zeta^{-1}$ (i.e., corresponding to the characters
indexed by $\pm1\in\Z/6\Z$):
\begin{equation}\label{eq:H1(E0)-decomp}
H^1(E_0)=H^1(E_0)^{(1)}\oplus H^1(E_0)^{(-1)}.
\end{equation}
Also, powers of $\Fr_r$ act on the two subspaces as $\<r\>$ acts on
$\{\pm1\}=(\Z/6\Z)^\times\subset\Z/6\Z$. 

More explicitly, if $r\equiv1\pmod6$, so that $\<r\>$ has two orbits
on $(\Z/6\Z)^\times$, then $\Fr_r$ preserves the two subspaces, and
the corresponding eigenvalues are
$$J(\{1\})=J_\Frr(\chi_{\Frr,6}^{-3},\chi_{\Frr,6}^{-2})\quad\text{and}\quad
J(\{-1\})=J_\Frr(\chi_{\Frr,6}^{3},\chi_{\Frr,6}^{2}),$$
where the Jacobi sums are as defined in equation \eqref{eq:def-J(o)}.

%
%

If $r\equiv5\pmod6$, so that $\<r\>$ has a unique orbit on
$(\Z/6\Z)^\times$, then $\Fr_r$ exchanges the two
subspaces, and the eigenvalues of $\Fr_r^2$ are both
$$J(\{1,-1\})=J_{\F_{r^2}}(\chi_{\F_{r^2},6}^{-3},\chi_{\F_{r^2},6}^{-2})
=J_{\F_{r^2}}(\chi_{\F_{r^2},6}^{3},\chi_{\F_{r^2},6}^{2}).$$

Finally, applying Lemma~\ref{lemma:charpoly}, we find that
$$\det\left(1-T\Fr_r\left|H^1(E_0)\right.\right)=
\prod_{o\in N_{r,6}}\left(1-J(o)T^{|o|}\right).$$

We remark that this result together with the values of $\ord(J(o))$
recorded in Section~\ref{ss:J(o)} are compatible with the well-known
fact that $E_0$ is ordinary if $p\equiv1\pmod6$ and supersingular if
$p\equiv-1\pmod6$.

\subsection{Artin--Schreier curves}\label{ss:Cnq}
For a positive integer $n$ relatively prime to $p$, let $C_{n,q}$ 
be the smooth projective curve over $\Frr$ defined by the equation
$$C_{n,q}:\quad u^n=t^q-t.$$
(We also use the equation $w^n=z^q-z$ when  more than one instance of
$C_{n,q}$ is under discussion.  Only $n=2,3,6$ will be used later in
this paper.)  The displayed equation defines a smooth affine curve, and
there is a unique point at infinity on $C_{n,q}$ which we denote by
$\infty\in C_{n,q}$.

The curve $C_{n,q}$ carries actions of $\mu_{n}$ via
$\zeta(t,u)=(t,\zeta u)$, and of $\Fq$ via $\alpha(t,u)=(t+\alpha,u)$.
(In fact, it carries an action of the larger group
$\Fq\sdp\mu_{n(q-1)}$, where $\zeta\in\mu_{n(q-1)}$ acts via
$\zeta(t,u)=(\zeta^nt,\zeta u)$.  In this section and the next, we
will only need the action of the subgroup $\mu_n\times\Fq$.  The
action of the larger group will be useful in Section~\ref{s:sha-alg}.)
The character group of $\mu_n\times\Fq$ is isomorphic to
$\Z/n\Z\times\Fq$.

The cohomology group $H^1(C_{n,q})$ has dimension $(q-1)(n-1)$, and
under the action of $\mu_n\times\Fq$, it decomposes into lines where
$\mu_n$ and $\Fq$ act through their non-trivial characters.  
(This is proven for $q=p$ in \cite[Cor.~2.2]{Katz81}, and the
arguments there generalize straightfowardly to the case $q=p^f$.)  In
particular, the subspace of $H^1(C_{n,q})$ where $\mu_n$ acts via a
given non-trivial character has dimension $q-1$, and the subspace
where $\Fq$ acts via a given non-trivial character has dimension
$n-1$.

Recall from Section~\ref{ss:orbits} that
$S=S_{n,q}:=\left(\Z/n\Z\setminus\{0\}\right)\times\Fqtimes$ and that
$O_{r,n,q}$ denotes the set of orbits of the action of $\<r\>$ on $S$.
We index the characters of $\mu_n\times\Fq$ (with values in the
coefficient field of our cohomology theory) which are non-trivial on
both factors by $S$.
The subspace of $H^1(C_{n,q})$ where $\mu_n\times\Fq$ acts via the
character indexed by $(i,\alpha)$ will be denoted by
$H^1(C_{n,q})^{(i,\alpha)}$. 
We thus obtain a direct sum decomposition of $H^1(C_{n,q})$ into lines as follows: 
\begin{equation}\label{eq:H1(Cnq)-decomp}
H^1(C_{n,q})=\bigoplus_{(i,\alpha)\in S_{n,q}}H^1(C_{n,q})^{(i,\alpha)}.    
\end{equation}
Katz \cite[Cor.~2.2]{Katz81} further gave a description of the action of
Frobenius on the cohomology $H^1(C_{n,q})$: the Frobenius $\Fr_r$ sends the subspace
indexed by $(i,\alpha)$ to the subspace indexed by
$(ri,\alpha^{1/r})$. 
If $o\in O_{r,n,q}$ is the orbit through
$(i,\alpha)$, then the $|o|$-th iterate $\Fr_r^{|o|}$ stabilizes the subspace
$H^1(C_{n,q})^{(i,\alpha)}$ (which is a line) and the eigenvalue of
  $\Fr_r^{|o|}$ on $H^1(C_{n,q})^{(i,\alpha)}$ is the Gauss sum
$$G(o):=G_{\F}(\chi_{\F,n}^i,\psi_\alpha)$$
where $\F=\F_{r^{|o|}}$.  (Again, Katz treated the case $q=p$, but the
generalization is straightforward.)  


Applying Lemma~\ref{lemma:charpoly}, we have
$$\det\left(1-T\Fr_r\left|H^1(C_{n,q})\right.\right)=
\prod_{o\in O_{r,n,q}}\left(1-G(o)T^{|o|}\right).$$

We remark that this result together with the values of $\ord(G(o))$
recorded in Section~\ref{ss:G(o)} are compatible with the well-known
fact that $C_{2,q}$ is supersingular, and they show that $C_{3,q}$ is
supersingular when $p\equiv-1\pmod6$ and neither supersingular nor
ordinary if $p\equiv1\pmod6$.  (In this last case, the slopes are
$1/3$ and $2/3$, both with multiplicity $q-1$, 
cf.~\cite[\S8.3]{PriesUlmer16}.)

\subsection{Fermat curves}\label{ss:Fd}
For a positive integer $d$ prime to $p$, let $F_d$ be the Fermat curve of
degree $d$ over $\Frr$.  This is by definition the smooth, projective
curve in $\P^2$ given by the homogeneous equation
$$F_d:\quad X_0^d+X_1^d+X_2^d=0.$$
The genus of $F_d$ is $(d-1)(d-2)/2$, so $H^1(F_d)$ has dimension
$(d-1)(d-2)$.  The curve $F_d$ carries an action of $(\mu_d)^3/\mu_d$
where the three copies of $\mu_d$ in the numerator act by
multiplication on the three coordinates, and the diagonally embedded
$\mu_d$ acts trivially.  Under the action of this group, $H^1(F_d)$
decomposes into lines on which each of the factors $\mu_d$ acts
non-trivially and the diagonally embedded $\mu_d$ acts trivially.
There are $(d-1)(d-2)$ such characters.  The action of Frobenius on
$H^1(F_d)$ is given by Jacobi sums.  Since we will not need the
cohomology of $F_d$ later in the paper, we omit the details.

\section{Domination by a product of curves}\label{s:domination}
In this section we define the Weierstrass and N\'eron models $\WW$ and
$\EE$ of $E$ and relate them to products of curves.  Throughout,
unless explicitly indicated otherwise by the notation, products of
varieties are over $\Frr$ (i.e., $\times$ means $\times_\Frr$).

Our ultimate aim is to compute the relevant part of the cohomology of
a model $\EE$ of $E$ by showing that $\EE$ is birational to the
quotient of a product of curves by a finite group.

\subsection{Models}
Let $\WW\to\P^1_{\Frr}$ be the Weierstrass model of $E$ over $K$, i.e., the
surface fibered over~$\P^1$ whose fibers are the plane cubic
reductions of $E$ at the places of $K$.  More precisely, let 
$$d=\deg(\omega_E)=\lceil q/6\rceil
=\begin{cases}
(q+5)/6&\text{if $q\equiv1\pmod6$}\\
(q+1)/6&\text{if $q\equiv5\pmod6$},
\end{cases}$$
and define $\WW$ by glueing the surfaces
$$y^2z=x^3+(t^q-t)z^3\subset \P^2_{x,y,z}\times\A^1_t$$
and
$$y^{\prime 2}z'=x^{\prime3}+(t^{\prime 6d-q}-t^{\prime
  6d-1})z^{\prime3}
\subset \P^2_{x',y',z'}\times\A^1_{t'}$$
via the map $([x',y',z'],t')=([x/t^{2d},y/t^{3d},z],1/t)$.  Then $\WW$
is a irreducible, normal, projective surface, and projection onto the
$t$ and $t'$ coordinates defines a morphism $\WW\to\P^1$ whose generic
fiber is $E$.

When $q\equiv5\pmod6$, $\WW$ is a regular surface (i.e., is smooth
over $\Frr$), and we define $\EE=\WW$.  When $q\equiv1\pmod 6$, $\WW$
has a singularity at the point $([x',y',z'],t')=([0,0,1],0)$ and is
regular elsewhere.  In this case, we define $\EE$ as the minimal
desingularization of $\WW$.  (The desingularization introduces 8 new
components.)

The reduction types of $\EE$ at closed points of $\P^1$ (i.e. at places of $K$) 
were recorded in Section~\ref{ss:reduction}.

\subsection{Sextic twists}\label{ss:sextic-DPC}
We saw above that $E$ becomes isomorphic to a constant curve after
extension of $K$ to $L=K[u]/(u^6=t^q-t)$.  Geometrically, this means
that $\EE$ is birational to a quotient of $E_0\times C_{6,q}$.  In
this subsection, we make this statement more explicit and deduce a
cohomological consequence.

Let $\mu_6$ act on $E_0\times C_{6,q}$ ``anti-diagonally,'' i.e., via
$\zeta(z,w,t,u)=(\zeta^2z,\zeta^3w,t,\zeta^{-1}u)$.  Define a rational
map $E_0\times C_{6,q}\ratto\WW$ by 
$$(z,w,t,u)\mapsto\left([x,y,z],t\right)=\left([zu^2,wu^3,1],t\right).$$
It is obvious that this map factors through the quotient
$\SS:=(E_0\times C_{6,q})/\mu_6$ and so we have a commutative diagram
\begin{equation*}
\xymatrix{\SS\ar@{-->}[r]\ar[d]&\WW\ar[d]\\
C_{6,q}/\mu_6\ar@{=}[r]&\P^1_t}
\end{equation*}
where the bottom horizontal arrow is the canonical isomorphism
$C_{6,q}/\mu_6\cong\P^1_t$ and the left vertical arrow is induced by
the projection onto $C_{6,q}$.

Now let $\tilde\SS\to\SS$ be a blow-up so that $\tilde\SS$ is smooth and
$\SS\ratto\WW$ induces a morphism $\tilde\SS\to\EE$.  (This can
be made completely explicit in terms of the fixed points of the action of
$\mu_6$ and the formula for the rational map
$E_0\times C_{6,q}\ratto\WW$, but the details will not be important
for our analysis.)  The diagram above then extends to
\begin{equation*}
\xymatrix{\tilde\SS\ar[r]\ar[d]&\EE\ar[d]\\
\SS\ar@{-->}[r]\ar[d]&\WW\ar[d]\\
C_{6,q}/\mu_6\ar@{=}[r]&\P^1_t.}
\end{equation*}

The following  encapsulates everything we need to know
about the geometry of $\tilde\SS\to\EE$.

\begin{propss} \mbox{}
  \begin{enumerate}
  \item The strict transform of $(O\times C_{6,q})/\mu_6$ in $\tilde\SS$
    maps to the zero section of $\EE$.
\item The strict transform of $(E_0\times \infty)/\mu_6$ in $\tilde\SS$
    maps to a fiber of $\EE\to\P^1$.
\item Every component of the exceptional divisor of $\tilde\SS\to\SS$
  maps into a fiber of $\EE\to\P^1$.
  \end{enumerate}
\end{propss}

\begin{proof}
  The first two points are obvious from the formula defining
  $E_0\times C_{6,q}\ratto\WW$.  The third point follows by examining the
  outer rectangle of the last displayed diagram.  Indeed, if $E$ is a
  component of the exceptional divisor of $\tilde\SS\to\SS$, then $E$
  lies over a single point of $C_{6,q}/\mu_6\cong\P^1_t$ and thus maps
  to a fiber of $\EE\to\P^1_t$.
\end{proof}

Let $T\subset H^2(\EE)$ be the subspace spanned by the
  classes of the zero section and components of fibers of
  $\EE\to\P^1$.  This is the subspace Shioda calls the ``trivial
  lattice'' (see \cite{Shioda92}).

\begin{corss}\label{cor:sextic-cohom}  
There is a canonical isomorphism
$$H^2(\EE)/T\cong \left(H^1(E_0)\tensor H^1(C_{6,q})\right)^{\mu_6}.$$
Here the exponent $\mu_6$ indicates the subspace invariant under the
anti-diagonal action of $\mu_6$.
\end{corss}

\begin{proof}
  The dominant morphism $\tilde\SS\to\EE$ induces a surjection
  $H^2(\tilde\SS)\to H^2(\EE)$.  Using the K\"unneth formula, taking
  invariants, and using the blow-up formula, we obtain a canonical
  isomorphism 
\begin{align*}
H^2(\tilde\SS)&\cong H^2(\SS)\oplus B\\
&\cong H^2(E_0\times C_{6,q} / \mu_6)\oplus B\\
&\cong H^2(E_0\times C_{6,q})^{\mu_6}\oplus B\\
&\cong \left(H^1(E_0)\tensor H^1(C_{6,q})\right)^{\mu_6}
\oplus \left(H^0(E_0)\tensor H^2(C_{6,q})\right)
\oplus \left(H^2(E_0)\tensor H^0(C_{6,q})\right)
\oplus B
\end{align*}
where $B$ denotes the subspace spanned by the classes of components of
the exceptional divisor of $\tilde\SS\to\SS$.

The proposition shows that $H^0(E_0)\tensor H^2(C_{6,q})$,
$H^2(E_0)\tensor H^0(C_{6,q})$, and $B$ all map to $T$.  Thus we have
a well-defined and canonical surjection
$$\left(H^1(E_0)\tensor H^1(C_{6,q})\right)^{\mu_6}\to H^2(\EE)/T.$$

To finish, we compare dimensions.  We recalled in
Section~\ref{s:curves} above that $\mu_6$ acts
on $H^1(E_0)$ through the characters $\zeta\mapsto\zeta^{\pm1}$, each
with multiplicity one (see equation \eqref{eq:H1(E0)-decomp}).  Similarly, $\mu_6$ acts on $H^1(C_{6,q})$
through characters $\zeta\mapsto\zeta^i$ with $i\not\equiv0\pmod6$,
each with multiplicity $q-1$ (see equation \eqref{eq:H1(Cnq)-decomp}).  Thus 
$$\dim\left(H^1(E_0)\tensor H^1(C_{6,q})\right)^{\mu_6}=2(q-1).$$

On the other hand, the Grothendieck--Ogg--Shafarevich formula says that $H^2(\EE)/T$ has
dimension $\deg(\NN_E)-4$ where $\NN_E$ denotes the conductor of $E$.
We noted above that $\deg(\NN_E)=2(q+1)$,
so $H^2(\EE)/T$ has dimension $2(q-1)$.  Therefore the surjection
$$\left(H^1(E_0)\tensor H^1(C_{6,q})\right)^{\mu_6}\to H^2(\EE)/T$$
is in fact a bijection.
\end{proof}


\subsection{Artin--Schreier quotients}\label{ss:AS-DPC}
In this subsection, we show that $\EE$ is birational to a quotient of
a product of Artin--Schreier curves, in the style of
\cite{PriesUlmer16}.

Let 
$$\CC=C_{2,q}:\quad w_1^2=z_1^q-z_1
\qquad \text{ and } \qquad
\DD=C_{3,q}:\quad w_2^3=z_2^q-z_2.$$ 
Write $\infty_\CC$ and $\infty_\DD$ for the points at infinity on $\CC$
and $\DD$ respectively.
Let $\Fq$ act on $\CC\times\DD$ ``diagonally,'' i.e., via 
$\alpha(z_1,w_1,z_2,w_2)=(z_1+\alpha,w_1,z_2+\alpha,w_2).$
It is easily seen that the sole fixed point of this action is $(\infty_\CC,\infty_\DD)$.

Define a rational map $\CC\times\DD\ratto\P^1_t$ by
$(z_1,w_1,z_2,w_2)\mapsto t=z_1-z_2$,
and a rational map $\CC\times\DD\ratto\WW$ by
$$(z_1,w_1,z_2,w_2)\mapsto([x,y,z],t)=([w_2,w_1,1],z_1-z_2).$$
Both of these maps are morphisms away from $(\infty_\CC,\infty_\DD)$,
and they clearly factor through the quotient $(\CC\times\DD)/\Fq$.

\begin{propss}
  There is a proper birational morphism $\SS'\to\CC\times\DD$ resolving
  the indeterminacy of $\CC\times\DD\ratto\WW$ such that the
  components of the exceptional divisor of $\SS'\to\CC\times\DD$ map either
  to the fiber of $\WW$ over $t=\infty$ or to the zero-section of $\WW$.
\end{propss}

\begin{proof}
  The proof of \cite[Prop~3.1.5]{PriesUlmer16} gives an explicit
  recipe for a morphism $\SS'\to\CC\times\DD$ resolving the
  indeterminacy of $\CC\times\DD\ratto\P^1_t$.  It is a sequence of four
  blow-ups of closed points.  Straightforward calculation, which we
  omit, shows that the induced map $\SS'\to\CC\times\DD\ratto\WW$ is in
  fact a morphism, and that it behaves as stated in the proposition on
  the components of the exceptional divisor.  Indeed, the first three
  blow-ups map to the fiber over $t=\infty$ and the last maps to the
  zero section.
\end{proof}

The diagonal action of $\Fq$ on $\CC\times\DD$ lifts uniquely to $\SS'$
and fixes the exceptional divisor pointwise.  It is clear that the
morphism $\SS'\to\WW$ factors through the quotient $\SS'/\Fq$, so we
have the following commutative diagram:
\begin{equation*}
\xymatrix{\SS'/\Fq\ar[r]\ar[d]&\WW\ar[d]\\
\P^1_t\ar@{=}[r]&\P^1_t.}
\end{equation*}

Now let $\tilde\SS\to\SS'/\Fq$ be a proper birational morphism so that
$\tilde \SS$ is a smooth projectve surface and the induced rational map
$\tilde\SS\ratto\EE$ is a morphism.  The diagram above then extends to
\begin{equation*}
\xymatrix{\tilde\SS\ar[r]\ar[d]&\EE\ar[d]\\
\SS'/\Fq\ar[r]\ar[d]&\WW\ar[d]\\
\P^1_t\ar@{=}[r]&\P^1_t.}
\end{equation*}

The following summarizes the relevant aspects of the geometry of this
picture.  

\begin{propss} \mbox{}
  \begin{enumerate}
  \item The strict transforms of $\infty_\CC\times\DD$ and
    $\CC\times\infty_\DD$ in $\tilde\SS$ map to the fiber of
    $\EE\to\P^1$ over $t=\infty$.
  \item The strict transforms in $\tilde\SS$ of the images in
    $\SS'/\Fq$ of the components of the exceptional fiber of
    $\SS'\to\CC\times\DD$ map to the fiber of $\EE\to\P^1$ over
    $t=\infty$ or to the zero-section of $\EE$.
\item Every component of the exceptional divisor of $\tilde\SS\to\SS'/\Fq$
  maps to a fiber of $\EE\to\P^1$.
  \end{enumerate}
\end{propss}

\begin{proof}
  The first point is obvious from the formula defining
  $\CC\times\DD\ratto\WW$.  The second point follows from the previous
  proposition.  The third point follows by examining the last displayed
  diagram.  Indeed, if $E$ is a component of the exceptional divisor
  of $\tilde\SS\to\SS/\Fq$, then $E$ lies over a single point of
  $\P^1_t$ and thus maps to a fiber of
  $\EE\to\P^1_t$.
\end{proof}

\begin{corss}\label{cor:AS-cohom} 
  Let $T\subset H^2(\EE)$ be the trivial lattice, i.e., the subspace
  spanned by the classes of the zero section and components of fibers
  of $\EE\to\P^1$.  There is a canonical isomorphism
$$H^2(\EE)/T\cong \left(H^1(\CC)\tensor H^1(\DD)\right)^{\Fq}.$$
Here the exponent $\Fq$ indicates the subspace invariant under the
diagonal action of $\Fq$.
\end{corss}

\begin{proof}
The proof is completely parallel to that of \ref{cor:sextic-cohom}, so we just sketch the argument.   
The dominant morphism $\tilde\SS\to\EE$ induces a surjection
  $H^2(\tilde\SS)\to H^2(\EE)$.  Using the K\"unneth formula, taking
  invariants, using the blow-up formula, and applying the proposition, we
 obtain a canonical 
 surjection
$$\left(H^1(\CC)\tensor H^1(\DD)\right)^{\Fq}\to H^2(\EE)/T.$$
To finish, we use Section~\ref{s:curves} and the proof of Corollary~\ref{cor:sextic-cohom} to check that 
$\left(H^1(\CC)\tensor H^1(\DD)\right)^{\Fq}$ and $H^2(\EE)/T$ both have
dimension $2(q-1)$.  Thus the displayed surjection is a bijection.
\end{proof}

\subsection{Fermat quotients}
The surfaces $\WW$ and $\EE$ have affine open subsets defined by an
equation with four monomials in three variables, namely
$$y^2=x^3+t^q-t.$$
In Shioda's terminology, these are ``Delsarte surfaces.''  This allows
one to show that (over a sufficiently large ground field) $\EE$ is
birational to a quotient of a Fermat surface by a finite group.  The
Fermat surface is itself birational to the quotient of a product of
two Fermat curves by a finite group.  Thus we arrive at a birational
presentation of $\EE$ as a quotient of a product of Fermat curves.  It
turns out that this presentation factors through the sextic twist
presentation given in Section~\ref{ss:sextic-DPC}, in a sense to be
explained below.  Thus, the Fermat quotient presentation does not give
essential new information, and we will only sketch the main points,
omitting most details.

Let $d=6q-6$.  Applying the method of Shioda (see \cite{Shioda86} and
\cite[\S6]{Ulmer07b} or \cite[Lecture 2, \S10]{Ulmer11}) yields a
dominant rational map from $F_d^2$ to $\EE$.  Explicitly, take two
copies of $F_d$ with homogeneous coordinates $[X_0,X_1,X_2]$ and
$[Y_0,Y_1,Y_2]$, and assume that $\Frr$ is large enough to contain a
primitive $2d$-th root of unity $\epsilon$.  Consider the rational map
$\phi:F_d^2\ratto\EE$ given by
$$\left([X_0,X_1,X_2],[Y_0,Y_1,Y_2]\right)\mapsto
(x,y,t)=\left(
\epsilon^2\frac{X_1^{2q-2}}{X_2^{2q-2}}\frac{Y_0^{2q-2}Y_1^2}{Y_2^{2q}},
\epsilon^{3q}\frac{X_0^{3q-3}}{X_2^{3q-3}}\frac{Y_0^{3q-3}Y_1^3}{Y_2^{3q}},
\epsilon^6\frac{Y_1^6}{Y_2^6}
\right).$$
Then it is not hard to check that $\phi$ is dominant of generic degree
$d^3$ and that it induces a birational isomorphism $F_d^2/G\ratto \EE$
where $G\subset\left(\mu_d^3/\mu_d\right)^2$ is the group generated by 
$$([1,1,\zeta],[\zeta,1,1]),\quad([\zeta^2,\zeta^3,1],[1,1,1]),\quad\text{and}
\quad([\zeta,\zeta^2,1],[1,\zeta^{q-1},1])$$
where $\zeta=\epsilon^2$ is a primitive $d$-th root of unity in
$\Frr$.  

Analyzing the geometry of $\phi$ would allow us to show that
$H^2(\EE)/T$ is isomorphic to a certain subspace of $H^2(F_d^2)$.
We omit the details, because, as we explain next, $\phi$ factors
through the rational map $E_0\times C_{6,q}\ratto\WW$ given in
Subsection~\ref{ss:sextic-DPC}. 

Indeed, consider the morphism $\tau_1:F_d\to E_0$ given by
$$[X_0,X_1,X_2]\mapsto(z,w)=
\left(\left(\frac{X_1}{X_2}\right)^{2q-2},
\left(\frac{\epsilon X_0}{X_2}\right)^{3q-3}\right)$$
and the morphism $\tau_2:F_d\to C_{6,q}$ given by
$$[Y_0,Y_1,Y_2]\mapsto(t,u)=
\left(\left(\frac{\epsilon Y_1}{Y_2}\right)^{6},
\frac{\epsilon Y_0^{q-1}Y_1}{Y_2^q}
\right).$$
Then it is straightforward to check that the diagram
\begin{equation*}
\xymatrix{F_d^2\ar^\phi@{-->}[rr]\ar_{\tau_1\times\tau_2}[rd]&&\EE\\
&E_0\times C_{6,q}\ar@{-->}[ur]}
\end{equation*}
commutes, where the right diagonal rational map is that given in
Subsection~\ref{ss:sextic-DPC}.  This implies that $H^2(\EE)/T$
already appears in the cohomology of $E_0\times C_{6,q}$, and
moreover, the relevant map is defined without requiring an extension
of $\Frr$.  We will thus omit any further consideration of Fermat curves.

\section{Geometric calculation of the $L$-function}\label{s:L-cohom}
In this section, we use the presentation of $\EE$ as a quotient of a
product of curves to give another calculation of $L(E,s)$ via the
cohomological formula for it proved in \cite{Shioda92}.  As in the
previous section, let $T\subset H^2(\EE)$ be the subspace spanned by
the classes of the zero-section and all components of all fibers of
$\EE\to\P^1$.  Shioda proved that
$$L(E,s)=\det\left(1-\Fr_rr^{-s}\left|H^2(\EE)/T\right.\right).$$

\subsection{Via sextic twists}
Recall from Section~\ref{ss:orbits} that $\<r\>$ acts on
$S^\times=(\Z/6\Z)^\times\times\Fqtimes$, the set of orbits being denoted $O_{r,6,q}^\times$.  As
in Section~\ref{ss:J(o)}, let $N_{r,6}$ denote the set of orbits of
$\<r\>$ on $(\Z/6\Z)^\times$, and let $\rho_6:O_{r,6,q}^\times\to N_{r,6}$ be
the map induced by the projection
$(\Z/6\Z)^\times\times\Fqtimes\to(\Z/6\Z)^\times$.
Define
$$n_6(o)=\frac{|o|}{|\rho_6(o)|}.$$
Note that $n_6(o)$ is either $|o|$ (if $r\equiv1\pmod6$) or $|o|/2$
(if $r\equiv-1\pmod6$).  To each orbit $o\in O_{r,6,q}^\times$ we
attach the Jacobi sum $J(\rho_6(o))$  (see Equation~\eqref{eq:def-J(o)}) and the Gauss sum $G(o)$ (see Equation~\eqref{eq:G(o)-def}).

\begin{thm}\label{thm:ST-L}
  $$L(E,s)=
\prod_{o\in O_{r,6,q}^\times}\left(1-J(\rho_6(o))^{n_6(o)}G(o)r^{-s|o|}\right).$$
\end{thm}

\begin{proof}
By Proposition~\ref{cor:sextic-cohom}, we know that
$$H^2(\EE)/T\cong\left(H^1(E_0)\tensor H^1(C_{6,q})\right)^{\mu_6}$$
where $\mu_6$ acts anti-diagonally.  Combining Equations~\eqref{eq:H1(E0)-decomp} and \eqref{eq:H1(Cnq)-decomp}, the right hand side decomposes as the direct sum
$$\bigoplus_{(i,\alpha)\in S^\times}
H^1(E_0)^{(i)}\tensor H^1(C_{6,q})^{(i,\alpha)}$$
where the summands are one-dimensional.  If $o\in O_{r,6,q}^\times$, then the
subspace 
$$\bigoplus_{(i,\alpha)\in o}
H^1(E_0)^{(i)}\tensor H^1(C_{6,q})^{(i,\alpha)}$$ is preserved by
the $r$-power Frobenius $\Fr_r$, and by what was recalled in Sections~\ref{ss:E0} and
\ref{ss:Cnq}, the eigenvalue of $\Fr_r^{|o|}$ on
$H^1(E_0)^{(i)}\tensor H^1(C_{6,q})^{(i,\alpha)}$ is
$J(\rho_6(o))^{n_6(o)}G(o)$.  By Lemma~\ref{lemma:charpoly}, the characteristic
polynomial of $\Fr_rr^{-s|o|}$ on the displayed subspace is
$\left(1-J(\rho_6(o))^{n_6(o)}G(o)r^{-s|o|}\right)$.
Taking the product over all orbits yields the theorem.
\end{proof}

\subsection{Via Artin--Schreier quotients}
As in Section~\ref{ss:orbits}, let $\<r\>$ act on
$S^\times=(\Z/n\Z)^\times\times\Fqtimes$ with orbits
$O_{r,n,q}^\times$.  For $n=2,3$, the natural projection $(\Z/6\Z)^\times\to(\Z/n\Z)^\times$ induces a map
$\pi_n:O^\times_{r,6,q}\to O^\times_{r,n,q}$.  Recall that we write 
$$m_2(o)=\frac{|o|}{|\pi_2(o)|}.$$
(There is no need for an analogous $m_3(o)$ since $|\pi_3(o)|=|o|$ for
all $o\in O_{r,6,q}^\times$.)
To each orbit $o\in O_{r,6,q}^\times$ we associate
Gauss sums $G(\pi_2(o))$ and $G(\pi_3(o))$ (see Section~\ref{ss:G(o)}).

\begin{thm}\label{thm:AS-L}
  $$L(E,s)=
\prod_{o\in O_{r,6,q}^\times}\left(1-G(\pi_2(o))^{m_2(o)}G(\pi_3(o))r^{-s|o|}\right).$$
\end{thm}

\begin{proof}
  By Corollary~\ref{cor:AS-cohom}, we have
$$H^2(\EE)/T\cong\left(H^1(C_{2,q})\tensor H^1(C_{3,q})\right)^{\Fq}$$
where $\Fq$ acts diagonally.  Using Equation~\eqref{eq:H1(Cnq)-decomp} twice, we get a direct sum decomposition of the right hand side: 
$$\bigoplus_{(i,\alpha)\in S^\times}
H^1(C_{2,q})^{(i\,\textrm{mod}\,2,\alpha)}\tensor H^1(C_{3,q})^{(i\,\textrm{mod}\,3,-\alpha)},$$
where all the summands are one-dimensional.  For any orbit $o\in O_{r,6,q}^\times$,  the
subspace 
\begin{equation*}
    \bigoplus_{(i,\alpha)\in o}
H^1(C_{2,q})^{(i\,\textrm{mod}\,2,\alpha)}\tensor H^1(C_{3,q})^{(i\,\textrm{mod}\,3,-\alpha)}
\end{equation*}
is preserved by the $r$-power Frobenius.
The results recalled in Section~\ref{ss:Cnq} show that the eigenvalue of
$\Fr_r^{|o|}$ acting on the line 
$H^1(C_{2,q})^{(i\,\textrm{mod}\,2,\alpha)}\tensor H^1(C_{3,q})^{(i\,\textrm{mod}\,3,-\alpha)}$ is $G(\pi_2(o))^{m_2(o)}G(\pi_3(o))$.
(Here we use that
$G_{\F}(\chi_{\F,3}^i,\psi_{-\alpha})=G_{\F}(\chi_{\F,3}^i,\psi_{\alpha})$, a consequence of the fact that
$-1$ is a cube in any finite field $\F$.)   
Lemma~\ref{lemma:charpoly} now implies that the characteristic polynomial
of $\Fr_rr^{-s|o|}$ on the displayed subspace is 
$\left(1-G(\pi_2(o))^{m_2(o)}G(\pi_3(o))r^{-s|o|}\right)$.
Taking the product over orbits then yields the theorem.
\end{proof}

\subsection{Comparison of $L$-functions}\label{ss:L-compare}
As a  check, we verify that the three expressions for $L(E,s)$
are in fact equal.

The ``Artin--Schreier'' expression for the $L$-function in
Theorem~\ref{thm:AS-L}  is visibly equal to the ``elementary''
expression in Theorem~\ref{thm:L-elem}.

The index sets for the products in the ``Artin--Schreier''  and ``sextic twist''
expressions for the $L$-function (Theorems~\ref{thm:AS-L} and
 \ref{thm:ST-L} respectively) are
the same, namely $O_{r,6,q}^\times$.  If $o\in O_{r,6,q}^\times$ is the orbit through
$(i,\alpha)$, let $o'$ be the orbit through $(-i,\alpha)$.  The map
$o\mapsto o'$ gives a bijection $O_{r,6,q}^\times\to O_{r,6,q}^\times$ with
$n_6(o)=n_6(o')$.  

Let $o\in O_{r,6,q}^\times$ and choose $(i,\alpha)\in o$.
Write $\F=\F_{r^{|o|}}$,
$\F'=\F_{r^{|\pi_2(o)|}}$, and $\F''=\F_{r^{|\rho_6(o)|}}$, so that
$\F/\F'$ is an extension of degree $m_2(o)$, and $\F/\F''$ is an
extension of degree $n_6(o)$.  Then
\begin{align*}
G(\pi_2(o))^{m_2(o)}G(\pi_3(o))
&=G_{\F'}(\chi_{\F',2}^i,\psi_\alpha)^{m_2(o)}G_\F(\chi_{\F,3}^i,\psi_\alpha)
&\text{(definition of $G(\pi_n(o)$)}\\
&=G_\F(\chi_{\F,2}^i,\psi_\alpha)G_\F(\chi_{\F,3}^i,\psi_\alpha)
&\text{(Hasse--Davenport relation)}\\
&=J_\F(\chi_{\F,2}^i,\chi_{\F,3}^i)G_\F(\chi_{\F,2}^i\chi_{\F,3}^i,\psi_\alpha)
&\text{(Equation~\eqref{eq:GaussJacobi})}\\
&=J_{\F''}(\chi_{\F'',2}^i,\chi_{\F'',3}^i)^{n_6(o)}G_\F(\chi_{\F,2}^i\chi_{\F,3}^i,\psi_\alpha)
&\text{(Hasse--Davenport relation)}\\
&=J(\rho_6(o'))^{n_6(o')}G_\F(\chi_{\F,2}^i\chi_{\F,3}^i,\psi_\alpha)\
&\text{(definition of $J(\rho_6(o'))$}\\
&&\text{and $n_6(o)=n_6(o')$)}\\
&=J(\rho_6(o'))^{n_6(o')}G_\F(\chi_{\F,6}^{-i},\psi_\alpha)
&\text{($2+3=-1\pmod6$)}\\
&=J(\rho_6(o'))^{n_6(o')}G(o')
&\text{(definition of $G(o')$).}
\end{align*}
Thus the $o$ factor in the ``Artin--Schreier'' product for $L(E,s)$ equals the $o'$
factor in the ``sextic twist'' product for $L(E,s)$.

\section{First application of the BSD conjecture}\label{s:BSD}
In this section, we show that the conjecture of Birch and
Swinnerton-Dyer (BSD) holds for $E$, and we deduce consequences for the
Mordell--Weil group $E(K)$.

\subsection{Notation and definitions}\label{ss:BSD-notations}
We recall the remaining definitions needed to state our BSD result.
There is a canonical $\Z$-bilinear pairing
$$\<\cdot,\cdot\>:E(K)\times E(K)\to\Q$$
which is non-degenerate modulo torsion.  (This is the canonical
N\'eron--Tate height pairing divided by $\log r$.  See \cite{Neron65} for the
definition and \cite[B.5]{HindrySilvermanDG} for a friendly
introduction.)  
Choosing a $\Z$-basis $P_1,\dots,P_R$ for $E(K)$ modulo
torsion, we define the \emph{regulator} of $E$ as
$$\Reg(E):=\left|\det\< P_i,P_j\>_{1\le i,j\le R}\right|.$$
The regulator is a positive rational number, well defined independently
of the choice of bases, and by convention, it is $1$ when the rank of
$E(K)$ is zero.

We write $H^1(K,E)$ for the \'etale cohomology of $K$ with
coefficients in $E$ and similarly for $H^1(K_v,E)$ for any place $v$ of $K$.  The
\emph{Tate--Shafarevich group} of $E$  is defined as
$$\sha(E):=\ker\left(H^1(K,E)\to\prod_v H^1(K_v,E)\right)$$
where the product is over the places of $K$ and the map is the product
of the restriction maps.  

The leading coefficient of the $L$-function (also called its \emph{special value} at $s=1$ or $T=r^{-1}$) is defined by 
\begin{equation*}
L^*(E):=\frac1{\rho!}
\left.\left(\frac{d}{dT}\right)^\rho L(E,T)\right|_{T=r^{-1}}
=\frac{1}{(\log r)^\rho}\frac1{\rho!}
\left.\left(\frac{d}{ds}\right)^\rho L(E,s)\right|_{s=1}
\end{equation*}
where $\rho$ is the order of vanishing $\rho:=\ord_{s=1}L(E,s)$.  
The point of the normalization by $(\log r)^{-\rho}$ is to ensure that $L^*(E)$ is a rational number (recall indeed that $L(E,s)$ is a polynomial with integral coefficients in $T=r^{-s}$).
Note that the above definition directly implies the two relations:
$$ L^\ast(E) = \left. \frac{L(E, T)}{(1-rT)^\rho}\right|_{T=r^{-1}} 
\qquad \text{ and } \qquad 
L^\ast(E) = \lim_{s\to 1} \frac{L(E, s)}{(1-r^{1-s})^\rho}.$$
We refer to Section~\ref{ss:definitions} for the definition of the local Tamagawa numbers $c_v$.

Here is our main result connecting all these invariants.

\begin{thm}\label{thm:BSD}
The BSD conjecture holds for $E$.  More precisely,
\begin{enumerate}
\item $\ord_{s=1}L(E,s)=\rk E(K)$,
\item $\sha(E)$ is finite,
\item we have an equality
$$L^*(E)=\frac{\Reg(E)|\sha(E)|\prod_vc_v}{r^{\deg(\omega_E)-1}|E(K)_{\mathrm{tors}}|^2}.$$
\end{enumerate}
\end{thm}

\begin{proof}
  This follows from the fact (Sections~\ref{ss:sextic-DPC} and
  \ref{ss:AS-DPC}) that the N\'eron model of $E$ is dominated by a
  product of curves and earlier work of Tate \cite{Tate66b} and Milne \cite{Milne75}.  See
  \cite[Thm.~9.1]{Ulmer11} for more details.
\end{proof}

As we have showed, the $L$-function $L(E, s)$ is a polynomial of degree $2(q-1)$ in $r^{-s}$. 
In particular, $\rho=\ord_{s=1}L(E, s)$ cannot exceed $2(q-1)$. By part (1) of the BSD result, this proves that
$ 0\leq\rk E(K) \leq 2(q-1)$. In what follows, we will describe more precisely the value of $\rk E(K)$, depending on $p\mod 6$.

We proved in Proposition~\ref{prop:Tawagawa-torsion} that
$\left|E(K)_{\mathrm{tors}}\right|=1$ and that $\prod_vc_v=1$, and we noted in
Section~\ref{ss:reduction}
that $\deg(\omega_E)=\lceil q/6\rceil$.  Thus the BSD
formula simplifies to
\begin{equation}\label{eq:BSD-simplified}
L^*(E)=\frac{\Reg(E)|\sha(E)|}{r^{\lfloor q/6\rfloor}}.  
\end{equation}

In the rest of this section, we will deduce consequences from part (1)
of the theorem, and in the following section we will use parts (2) and
(3).

\subsection{Explicit $L$-function for
  $p\equiv1\pmod6$}\label{ss:L-p=1(6)} 
Recall that we have shown that
$$L(E,T)=
\prod_{o\in O_{r,6,q}^\times}\left(1-G(\pi_2(o))^{m_2(o)}G(\pi_3(o))T^{|o|}\right)$$
where we substitute $T$ for $r^{-s}$.  We will make this more explicit
using results from Section~\ref{ss:explicit-sums}.

First, note that when $p\equiv1\pmod6$, the action of $\<r\>$ on
$(\Z/6\Z)^\times$ is trivial, so an orbit
$o\in O_{r,6,q}^\times$ consists of pairs $(i,\alpha)$ where
$i\in(\Z/6\Z)^\times$ is constant and $\alpha\in\Fqtimes$ runs through an orbit
$\overline o\in O_{r,q}$ (recall that $O_{r,q}$ denotes the set of orbits of the action of $\<r\>$ on $\Fqtimes$).  In
particular, we have $|\pi_2(o)|=|o|$ so that $m_2(o)=1$.

For a given orbit $\overline o\in O_{r,q}$, let us consider the two orbits in $O^\times_{r,6,q}$ 
$$o=\{(1,\alpha) : \alpha\in\overline o\}
\quad\text{and}\quad
o'=\{(-1,\alpha) : \alpha\in\overline o\}$$ 
``lying over $\overline o$'' and the two corresponding factors in the product for the $L$-function. 
Set $\F=\Frr(\alpha)$ and note that $\F$ is an
extension of $\Frr=\F_{p^{\nu}}$ of degree $|o|=|o'| = |\overline{o}|$.
By definition we have
\begin{multline}\label{eq:exp-L-p=1(6)}
\left(1-G(\pi_2(o))G(\pi_3(o))T^{|o|}\right)
\left(1-G(\pi_2(o'))G(\pi_3(o'))T^{|o'|}\right)\\
=\left(1-G_\F(\chi_{\F,2},\psi_\alpha)
  G_\F(\chi_{\F,3},\psi_\alpha)T^{|o|}\right)
\left(1-G_\F(\chi_{\F,2},\psi_\alpha)
  G_\F(\chi_{\F,3}^{-1},\psi_\alpha)T^{|o|}\right) =: L_{\overline o}(T).  
\end{multline}

Since $|\F|=p^{\nu|o|}$, it follows from
Equation~\eqref{eq:quad-Gauss-sum} that
$$\ord\left(G_\F(\chi_{\F,2},\psi_\alpha)\right)=\frac12\nu|o|.$$
On the other hand,  Equation~\eqref{eq:cubic-Gauss-sum1} yields that
$$\ord\left(G_\F(\chi_{\F,3},\psi_\alpha)\right)=\frac23\nu|o|
\quad\text{and}\quad
\ord\left(G_\F(\chi_{\F,3}^{-1},\psi_\alpha)\right)=\frac13\nu|o|.$$
In particular, the inverse roots of the product $L_{\overline o}(T)$ have valuation $(7/6)\nu$ and   $(5/6)\nu$. 
We deduce that $T= r^{-1}$, which satisfies $\ord(r^{-1}) = -\nu $, cannot be a root of $L_{\overline o}(T)$.

Since this holds for any orbit $\overline o\in O_{r,q}$ and since $L(E, T) = \prod_{\overline o\in O_{r,q}}L_{\overline o}(T)$, we obtain that $L(E, T)$ does not vanish at $T=r^{-1}$.
This establishes the first two points of the following result.

\begin{propss}\label{prop:ord-L-p=1(6)}
Assume that $p\equiv 1\bmod{6}$.
\begin{enumerate}
 \item The inverse roots on the right hand side of
   Equation~\eqref{eq:exp-L-p=1(6)} have valuations $(7/6)\nu$ and
   $(5/6)\nu$. 
\item $\ord_{s=1}L(E,s)=0$.
\item $E(K)=0$.
\item $\Reg(E)=1$.
\end{enumerate}
\end{propss}

\begin{proof}
Points (1) and (2) follow immediately from the above discussion.
It then follows from our BSD result
  (Theorem~\ref{thm:BSD}) that $\rk E(K)=0$ so that $E(K)$ is torsion. But we showed in
  Proposition~\ref{prop:Tawagawa-torsion} that $E(K)_{\mathrm{tors}}=0$, so  $E(K)=0$.
  Finally, since $E(K)$ has rank 0, the regulator is~$1$.
\end{proof}

We remark that point (1) of the proposition leads to another proof of
BSD in this case. Indeed, the inequality $0\leq \rk E(K)\leq \ord_{s=1} L(E,s)$ is
known in general (see \cite{Tate66b}), so if $\ord_{s=1}L(E, s)= 0$, then
$\rk E(K) = \ord_{s=1} L(E,s)=0$, and this equality between algebraic and analytic ranks implies
the rest of the BSD conjecture (by the main result of \cite{KatoTrihan03}).

\subsection{Explicit $L$-function for
  $p\equiv-1\pmod6$}\label{ss:L-p=-1(6)}  
As in the preceding subsection, we start from the expression
$$L(E,T)=
\prod_{o\in
  O_{r,6,q}^\times}\left(1-G(\pi_2(o))^{m_2(o)}G(\pi_3(o))T^{|o|}\right),$$
which we make more explicit, in the case when $p\equiv-1 \pmod 6$, using results from
Section~\ref{ss:explicit-sums}.

Let $o\in O^\times_{r,6,q}$ be an orbit,  pick $(i, \alpha)\in o$ and write $\F=\F_{r^{|o|}}$.
If $m_2(o)=1$  then, by definition of the Gauss sums, we have
$$\left(1-G(\pi_2(o))^{m_2(o)}G(\pi_3(o))T^{|o|}\right)
=\left(1-G_\F(\chi_{\F,2},\psi_\alpha)
  G_\F(\chi_{\F,3}^i,\psi_\alpha)T^{|o|}\right).$$
On the other hand, if $m_2(o)=2$, i.e., if $|o|=2|\pi_2(o)|$, 
then  setting $\F'=\Frr(\alpha)=\F_{r^{|\pi_2(o)|}}$ 
(which is  a quadratic extension of $\F$), 
the Hasse--Davenport relation yields
$$G(\pi_2(o))^{m_2(o)}=G_{\F'}(\chi_{\F',2},\psi_\alpha)^2
=G_\F(\chi_{\F,2},\psi_\alpha).$$
Thus, in both cases, we can rewrite the factor of $L(E,T)$ indexed by $o\in O_{r,6,q}^\times$ as 
$$\left(1-G(\pi_2(o))^{m_2(o)}G(\pi_3(o))T^{|o|}\right)
=\left(1-G_\F(\chi_{\F,2},\psi_\alpha)
  G_\F(\chi_{\F,3}^i,\psi_\alpha)T^{|o|}\right)$$
where $\F=\F_{r^{|o|}}$ and $(i,\alpha)\in o$.
Now using Equations~\eqref{eq:quad-Gauss-sum} and
\eqref{eq:cubic-Gauss-sum2} and recalling that $\Frr=\F_{p^\nu}$, we remark that
\begin{equation*} 
 G_\F(\chi_{\F,2},\psi_\alpha)
  G_\F(\chi_{\F,3}^i,\psi_\alpha) 
= p^{*\nu|o|/2}\chi_{\F,3}^{-i}(\alpha)(-p)^{\nu|o|/2} 
=\epsilon_o r^{|o|},
\end{equation*}
where $\epsilon_o$ is a 6th root of unity, namely
\begin{equation}\label{eq:epsilon-o}
\epsilon_o=(-1)^{(p+1)\nu|o|/4}\chi_{\F,3}^{-i}(\alpha).    
\end{equation}
Note that, $p$ being odd and $\nu |o|$ being even, the exponent $(p+1)\nu|o|/4$ of $-1$ is an integer.
Therefore, for any orbit $o\in O_{r,6,q}^\times$, 
the factor of $L(E,T)$ indexed by $o$ can be rewritten as 
\begin{equation}\label{eq:exp-L-p=-1(6)}
\left(1-G(\pi_2(o))^{m_2(o)}G(\pi_3(o))T^{|o|}\right) = \left(1-\epsilon_o r^{|o|}T^{|o|}\right).
\end{equation}
We can now prove the following result, analogous to Proposition~\ref{prop:ord-L-p=1(6)}.

\begin{propss}\label{prop:ord-L-p=-1(6)}
Assume that $p\equiv-1\pmod6$. Let 
\begin{equation*}
\rho = \rho_{r,q} := \left|\left\{o\in O^\times_{r,6,q} : (p+1)\nu |o|\equiv 0\bmod{8} \text{, and } \alpha \text{ is a cube in } \F^\times_{r^{|o|}} \text{ for any }(i,\alpha)\in o\right\}\right|.
\end{equation*} 
Then
\begin{enumerate}
\item $\ord_{s=1}L(E,s)=\rho$.
\item $E(K)$ is free abelian of rank $\rho$.
\item For a given $q$,  $\rk E(K)=2(q-1)$ for $\Frr$ sufficiently large.  More
  precisely, if $r=p^\nu$ is a power of $q$, $(p+1)\nu\equiv0\pmod8$, and
  $3(q-1)|(r-1)$, then $\rk E(K)=2(q-1)$.
\item For a given $r$, $\rk E(K)$ is unbounded as $q$ varies. Indeed, for every $\epsilon>0$, 
if $q=p^f$ and $f$ is a sufficiently large multiple of 4, then $\rk E(K)> 2(1-\epsilon)p^f/f$. 
\end{enumerate}
\end{propss}

\begin{proof}
  By our formula for $L(E, s)$ and Equation~\eqref{eq:exp-L-p=-1(6)}, the order of vanishing of $L(E,s)$ at $s=1$ is equal to the number
  of orbits $o\in O_{r,6,q}$ such that
  $G(\pi_2(o))^{m_2(o)}G(\pi_3(o))=r^{|o|}$, i.e., the number of orbits such that $\epsilon_o=1$.  
  Part (1) then follows easily from Equation \eqref{eq:epsilon-o}. 
  For (2), it follows from the BSD theorem 
  (Theorem~\ref{thm:BSD}) that $\rk E(K)=\rho$, and we showed in
  Proposition~\ref{prop:Tawagawa-torsion} that $E(K)_{\mathrm{tors}}=0$, so that $E(K)$ is indeed
  free abelian of rank $\rho$.  
  The conditions in (3) guarantee that all orbits $o$ have size 1 and satisfy $\epsilon_o=1$.  In this case, there are $2(q-1)$ orbits, all contributing to $\rho$, and this yields the claim.
(Under these assumptions, the $L$-function of $E$ therefore admits a very simple expression: $L(E, s) = (1-r^{1-s})^{2(q-1)}$).
  
  To prove (4), we first note that it suffices to treat the case $r=p$,
i.e., $\nu=1$.  Next, we note that ``most'' elements $\alpha\in\F_{p^f}$
satisfy $\Fp(\alpha)=\F_{p^f}$.  Indeed, it is elementary that the
number of elements in $\F_{p^f}$ which do not lie in a smaller field
is at least $p^f-(\log_2f)p^{f/2}$.  It follows that for every
$\epsilon>0$, there is a constant $f_0$ such that
$$\left|\left\{\alpha\in\F_{p^f}\left|\Fp(\alpha)=\F_{p^f}\right.\right\}\right|
\ge (1-\epsilon)p^f$$ for all $f>f_0$.  On the other hand, at least
$(1/3)(p^f-1)$ elements of $\F_{p^f}^\times$ are cubes.  Thus, if
$\epsilon<1/3$, then for all sufficiently large $f$, the number of
elements of $\F_{p^f}$ which are cubes and which generate $\F_{p^f}$
is at least $(1/3-\epsilon)p^f$.  If $f$ is even and $\alpha$ has
these properties, then the orbit through $(i,\alpha)$ has size $f$,
and if $f$ is a multiple of $4$, then these orbits all contribute to
$\rho$.  This shows that for $f$ divisible by $4$ and sufficiently large, $\rho$
is bounded below by $2(1-\epsilon)p^f/f$, and this completes the proof of
part (4) of the Proposition.
\end{proof}




We note that although the rank is always unbounded for varying $q$, it does not go to infinity with $q=p^f$, 
i.e., the rank of $E(K)$ may be small even when $f$ is large.  
For example, when $p\equiv5\pmod{12}$ and $\nu=1$, it follows from part (1) of the Proposition that the rank is zero for all odd $f$.

\section{$p$-adic size of $L^*(E)$ and $\sha(E)$}\label{s:BSD-p}
The  special value $L^*(E)$ was defined in the previous section.
Since $L(E,T)$ is a polynomial in $T$ with integer coefficients,
$L^*(E)$ actually lies in $\Z[1/p]$.  In this section, we use the explicit
presentation of the $L$-function in terms of exponential sums to
estimate the $p$-adic valuation of $L^*(E)$, and then use the BSD
formula to deduce consequences for $\Reg(E)|\sha(E)|$.

Recall from Section~\ref{ss:finite-fields} that we fixed a prime ideal $\gP$ of $\bar{\Z}$ which lies over $p$.  
As before, we denote by $\ord$ the $p$-adic valuation of $\Qbar$ associated to $\gP$ normalized
so that $\ord(p)=1$.

\begin{prop}\label{prop:ord-L^*} 
Given data $p,q$ and $r=p^{\nu}$ as before, we have
  \mbox{}
  \begin{enumerate}
  \item If $p\equiv1\pmod6$, 
$$\ord(L^*(E))=-\frac{q-1}6\nu.$$
\item If $p\equiv-1\pmod6$, then $L^*(E)$ is
  an integer, so $\ord(L^*(E))\ge0$.
\item If $p\equiv-1\pmod6$ and $r$ is sufficiently large \textup{(}in the sense
  of Prop.~\ref{prop:ord-L-p=-1(6)} (3)\textup{)},   
$L^*(E)=1$.
  \end{enumerate}
\end{prop}

\begin{proof}
  First assume that $p\equiv1\pmod6$.  As we saw in Section~\ref{ss:L-p=1(6)}, $L^\ast(E)$ is simply the value of $L(E, T)$ at $T=r^{-1}$. 
  We further 
  showed that $L(E,T)$ is the product over orbits $\overline o$ of $\<r\>$ acting on
  $\Fqtimes$ of factors of the form
$$\left(1-\gamma_1 T^{|\overline o|}\right)
\left(1-\gamma_2 T^{|\overline o|}\right)$$ where
$\ord(\gamma_1)=(5/6)\nu|\overline o|$ and
$\ord(\gamma_2)=(7/6)\nu|\overline o|$.  (See 
Proposition~\ref{prop:ord-L-p=1(6)} (1)
and the discussion above that result.)  Substituting
$T=r^{-1}=p^{-\nu}$, we see that the contribution to $\ord L^\ast(E)$ from the pair of
factors associated to $\overline o$ has valuation $(-1/6)\nu|\overline
o|$.  Taking the product over all orbits shows that
$$\ord(L^*(E))= \sum_{\overline o \in O_{r,q}} -\frac{\nu |\overline o|}{6} = -\frac{\nu}{6} \sum_{\overline o \in O_{r,q}} |\overline o|= -\frac{(q-1)\nu}{6},$$
and this establishes part (1) of the
proposition.  

Now assume that $p\equiv-1\pmod6$.  In Section~\ref{ss:L-p=-1(6)}, we
showed that $L(E,T)$ is the product over orbits $o\in
O_{r,6,q}^\times$ of factors of the form
$\left(1-\epsilon_or^{|o|}T^{|o|}\right)$
where $\epsilon_o$ is a 6th root of unity.  If $\epsilon_o\neq1$, then
the contribution of this factor to the special value is
$(1-\epsilon_o)$, an algebraic integer.  If $\epsilon_o=1$, then the
contribution is
$$\left.\frac{(1-r^{|o|}T^{|o|})}{1-rT}\right|_{T=r^{-1}}
=\left.\left(1+rT+\cdots+(rT)^{|o|-1}\right)\right|_{T=r^{-1}}=|o|,$$
an integer.  This shows that $L^*(E)$ is an algebraic integer, and since
it also lies in $\Z[1/p]\subset\Q$, $L^\ast(E)$ is an integer.  This establishes part (2) of the
proposition.
For part (3), we note that if $r$ is sufficiently large, all orbits
$o$ are singletons and all the $\epsilon_o$ are $1$ (see Proposition~\ref{prop:ord-L-p=-1(6)}(3)). 
The analysis of the preceding paragraph shows that $L^*(E)=1$.
\end{proof}

Now we apply the BSD formula, as simplified in
Equation~\eqref{eq:BSD-simplified}:
$$L^*(E)=\frac{\Reg(E)|\sha(E)|}{r^{\lfloor q/6\rfloor}}.$$

\begin{cor}\label{cor:p-parts-via-L}
  \mbox{}
  \begin{enumerate}
  \item If $p\equiv1\pmod6$, then 
$$\Reg(E)=1\quad\text{and}\quad\ord(|\sha(E)|)=0.$$ 
In particular, the $p$-primary part of $\sha(E)$ is trivial.
\item if $p\equiv-1\pmod6$, then 
$$\ord(\Reg(E)|\sha(E)|)\ge \lfloor q/6\rfloor\nu.$$
\item If $p\equiv-1\pmod6$ and $r$ is sufficiently large \textup{(}in the sense
  of Prop.~\ref{prop:ord-L-p=1(6)} (4)\textup{)},   
then
$$\Reg(E)|\sha(E)|=r^{\lfloor q/6\rfloor}=p^{\nu\lfloor q/6\rfloor}.$$
In particular, $\sha(E)$ is a $p$-group.
  \end{enumerate}
\end{cor}

\begin{proof}
   If $p\equiv1\pmod6$, then combining Proposition~\ref{prop:ord-L^*}
  with the BSD formula~\eqref{eq:BSD-simplified} yields that
  $$\ord(\Reg(E)|\sha(E)|)=0.$$  
  We showed in Proposition~\ref{prop:ord-L-p=1(6)} that $\Reg(E)=1$,
  so that $\ord(|\sha(E)|)=0$.  This proves part (1).

If $p\equiv-1\pmod6$, then Proposition~\ref{prop:ord-L^*}  says that
$L^*(E)$ is an integer, and it follows immediately from
formula~\eqref{eq:BSD-simplified} that
$\ord(\Reg(E)|\sha(E)|)\ge\lfloor q/6\rfloor\nu$.  This yields part (2).

For part (3), we know from Proposition~\ref{prop:ord-L^*}  that $L^*(E)=1$, so
formula~\eqref{eq:BSD-simplified} implies that
$\Reg(E)|\sha(E)|=r^{\lfloor q/6\rfloor}$.
By \cite[Prop.~3.1.1]{UlmerBS}, $\Reg(E)$ is an integer, so both it
and $|\sha(E)|$ are powers of $p$.  This establishes part (3).
\end{proof}

Following \cite[\S4]{UlmerBS}, let us consider the limit
$$ \dim\sha(E):=\lim_{n\to\infty}\frac{\log\left|\sha(E\times\F_{r^n}(t))[p^\infty]\right|} {\log(r^n)},$$
where 
$\sha(-)[p^\infty]$ denotes the $p$-primary part of $\sha(-)$.
As is shown in \emph{loc.\ cit.}, the limit exists and is a non-negative integer, called the ``dimension of $\sha$'' of $E$.
The value of $\dim\sha(E)$ is expressed in terms of the valuations of the inverse roots of $L(E,T)$ in
\cite[Prop.~4.2]{UlmerBS}. 

In the situation at hand, the mentioned expression and the results of Sections \ref{ss:L-p=1(6)} and \ref{ss:L-p=-1(6)} directly yield the following values for $\dim\sha(E)$:

\begin{cor}\label{cor:dimsha-via-L}
\mbox{}
\begin{enumerate}
  \item If $p\equiv1\pmod6$, then $\dim\sha(E) = 0$.
  \item If $p\equiv-1\pmod6$, then $\dim\sha(E) = \lfloor q/6\rfloor$.
\end{enumerate}
\end{cor}

\section{Algebraic analysis of $\sha(E)[p^\infty]$}\label{s:sha-alg}
In this section we recover the results of
Corollaries~\ref{cor:p-parts-via-L} and~\ref{cor:dimsha-via-L} regarding the $p$-torsion in
$\sha(E)$ by algebraic means, more specifically via crystalline
cohomology.  Here is the statement.

\begin{prop}\label{prop:sha-alg}
\mbox{}
\begin{enumerate}
\item If $p\equiv1\pmod6$, then $\sha(E)[p]=0$.
\item If $p\equiv-1\pmod6$, then $\dim\sha(E)=\lfloor q/6\rfloor$.
\end{enumerate}
\end{prop}

The proof will use that the N\'eron model $\EE$ is dominated by the
product of curves $E_0\times C_{6,q}$, knowledge of the crystalline
cohomology of the curves, and $p$-adic semi-linear algebra, as in
\cite[\S6-8]{UlmerBS}.   We collect the needed background results in
the next subsection and treat the cases $p\equiv1\pmod6$ and
$p\equiv-1\pmod6$ separately in the following two subsections.  

\subsection{Preliminaries}\label{ss:prelims}
Let $W=W(\Frr)$ denote the ring of Witt vectors over $\Frr$ and $\sigma$ denote its Frobenius morphism.
We denote the Dieudonn\'e ring by $A=W\{F,V\}$: this is the non-commutative polynomial ring over $W$ with indeterminates $F,V$ modulo the relations $FV=VF=p\in W$,  $Fw = \sigma(w)F$, and  $\sigma(w)V = Vw$ for all $w\in W$.

Throughout this section, we  write $H^1(C)$ for the integral
crystalline cohomology $H^1_{crys}(C/W)$ of a curve $C$ over $\Frr$.
This is a finitely generated, free $W=W(\Frr)$-module equipped with
semi-linear actions of $F$ and $V$ such that $FV=VF=$ multiplication by
$p$.  In other words, $H^1(C)$ is a module over the Dieudonn\'e ring $A$.
We will apply this for $C=E_0$ and $C=C_{6,q}$ and make it much more explicit below.

We saw in Section~\ref{ss:sextic-DPC} that  the N\'eron model $\EE$ of
$E$, is birational to the quotient of $E_0\times C_{6,q}$ by the
anti-diagonal action of $\mu_6$.  Then \cite[Prop.~6.2]{UlmerBS} says that
\begin{align}
\sha(E)[p^\infty]
&\cong\Br(\EE)[p^\infty]\notag\\
&\cong\Br((E_0\times C_{6,q})/\mu_6)[p^\infty]\label{eq:sha-Br}\\
&\cong\Br(E_0\times C_{6,q})[p^\infty]^{\mu_6}\notag
\end{align}
where the exponent indicates the invariant subgroup.  Moreover,  by
\cite[Prop.~6.4]{UlmerBS}, for all $n\geq 1$  we have
\begin{equation}\label{eq:Br-Hom}  
\Br(E_0\times C_{6,q})[p^n]\cong
\frac{\Hom_A\left(H^1(E_0)/p^n,H^1(C_{6,q})/p^n\right)}
{\Hom_A\left(H^1(E_0),H^1(C_{6,q})\right)/p^n}
\end{equation}
compatibly with the action of $\mu_6$.

To prove part (1) of the proposition, we will show that the
$\mu_6$-invariant part of the numerator in the last expression is zero
whenever $p\equiv1\pmod6$.  For part (2), we will recall from
\cite[\S8]{UlmerBS} that the growth of
$\sha(E\times\F_{r^m}(t))[p^\infty]$ as a function of $m$ is controlled
by the numerator in the previous display, and this is in turn
computable in terms of the action of $\<p\>$ on a finite set indexing
the cohomology of $E_0$ and $C_{6,q}$.  

\subsection{Explicit $A$-module structure of $H^1(E_0)$ and $H^1(C_{6,q})$}
We now make explicit the results on the cohomology groups $H^1(E_0)$ and $H^1(C_{6,q})$ (viewed as $A$-modules) that will be needed below.  
All results stated in  this subsection follow from well-known results about Fermat curves and their quotients, as recalled in \cite[\S7]{UlmerBS} and in \cite{Katz81}.

Let 
$I=\{\pm1\}\subset(\Z/6\Z)^\times=I_0\cup I_1$
where $I_0=\{1\}$ and $I_1=\{-1\}$.  As a $W$-module, $H^1(E_0)$
has rank two and is generated by classes $e_i$ with $i\in I$, where
$e_{-1}$ is the class of the regular differential $dx/y$ and $e_1$ is
associated to the meromorphic differential $xdx/y$.  (This can be
taken to mean that the restriction of $e_1$ to $E_0\setminus\{O\}$ is
the class of the regular differential $xdx/y$.)  The indexing is
motivated by the fact that over an extension of $\Frr$ large enough to
contain the 6-th roots of unity, one has
$$\zeta^*(e_1)=\zeta e_1
\quad\text{and}\quad
\zeta^*(e_{-1})=\zeta^{-1} e_{-1}$$
for all $\zeta\in\mu_6$,
where the $\zeta$s on the left of each equation are in the finite field $\Frr$
and those on the right are their Teichm\"uller lifts to the Witt
vectors $W$.  The action of $A$ satisfies $F(e_i)=c_i e_{pi}$
for some $c_i\in W$ with 
\begin{equation}\label{eq:ordci}
\ord(c_i)=\begin{cases}
0&\text{if $i\in I_0$}\\
1&\text{if $i\in I_1$.}
\end{cases}
\end{equation}
Since $FV=p$, we deduce that $V(e_i)=p/\sigma^{-1}(c_{i/p})e_{i/p}$. 

Let $J\subset\Z/6(q-1)\Z$ be the set of classes which are non-zero
modulo $6$.  Given $j\in J$, there is a unique pair of integers
$(a,b)$ with $1\le a\le q-1$, $1\le b\le5$, and
$j\equiv 6a-b\pmod{6(q-1)}$.  Then $H^1(C_{6,q})$ is a free $W$-module
of rank $5(q-1)$ with basis elements $f_j$, $j\in J$, where $f_j$ is
associated to the differential $t^{a-1}dt/u^b$ in the following sense:
Let $J_1\subset J$ be the set of classes $j$ whose associated $(a,b)$
satisfy $a<qb/6$.  For these $j$, the differential $t^{a-1}dt/u^b$ is
everywhere regular on $C_{6,q}$ and $f_j$ is its class.  Let
$J_0=J\setminus J_1$.  If $j\in J_0$, the differential $t^{a-1}dt/u^b$
is regular on $C_{6,q}\setminus\{\infty\}$, and the restriction of
$f_j$ to the open curve is the class of this differential.  Over an
extension of $\Frr$ large enough to contain the roots of unity of order
$6(q-1)$, we have $\zeta^*f_j=\zeta^jf_j$ for all $\zeta\in\mu_{6(q-1)}$ (with the same convention as before).
The action of $A$ on $H^1(C_{6,q})$ is given by $F(f_j)=d_jf_{pj}$, for some $d_j\in W$ satisfying 
$$\ord(d_j)=\begin{cases}
0&\text{if $j\in J_0$}\\
1&\text{if $j\in J_1$.}
\end{cases}
$$
Since $FV=p$, we obtain that $V(f_j)=p/\sigma^{-1}(c_{j/p})f_{j/p}$. 

Fix $j\in J$ with $j\not\equiv0\pmod3$.  Let $\F=\Frr(\mu_{6(q-1)})$
and let $m=[\F:\Fp]$, so that $p^mj\equiv j\pmod{6(q-1)}$.  Then
the $m$th power $F^m$ of the Frobenius 
acts on $f_j$ by multiplication by a Gauss sum.  More precisely, let
$\chi = \chi_{\F,6(q-1)}$ be the character defined in
Section~\ref{ss:mult-chars}, viewed as a $W$-valued character.  Then $F^m f_j=G_jf_j$ 
where
$G_j=G_\F(\chi^j,\psi_1)$.  
When $p\equiv1\pmod6$, it follows from Stickelberger's theorem that
\begin{equation}\label{eq:ordGj}
\ord(G_j)=\begin{cases}
\frac23m&\text{if $j\equiv1\pmod3$}\\
\frac13m&\text{if $j\equiv2\pmod3$.}
\end{cases}
\end{equation}
(This is essentially the same calculation as that in
Section~\ref{ss:explicit-sums}.) 

When $p\equiv1\pmod6$, we will calculate
$\Hom_A(H^1(E_0)/p,H^1(C_{6,q})/p)$ explicitly in the next subsection
and see that it vanishes.  In the following subsection, we will assume
$p\equiv-1\pmod6$ and use the action of $\<p\>$ on $I\times J$ to
compute $\dim\sha(E)$ as in \cite[\S8]{UlmerBS}.

\subsection{Proof of Proposition~\ref{prop:sha-alg} part (1)}
In light of the isomorphisms in equations~\eqref{eq:sha-Br} and
\eqref{eq:Br-Hom}, to show that $\sha(E)[p]=0$ in the case when $p\equiv1\pmod 6$, it  suffices to
show that
$$\Hom_A\left(H^1(E_0)/p,H^1(C_{6,q})/p\right)^{\mu_6}=0.$$
To that end, let
$\varphi\in\Hom_A\left(H^1(E_0)/p,H^1(C_{6,q})/p\right)^{\mu_6}$.  Since $\varphi$ is, in particular, a $W$-linear map, we can write
$$\varphi(e_i)=\sum_{j}\alpha_{i,j}f_j$$
for all $i\in I=(\Z/6\Z)^\times$, where the sum runs over $j\in J\subset
\Z/6(q-1)\Z$, and $\alpha_{i,j}\in W/p=\Frr$. 
In order that $\varphi$ commute
with the anti-diagonal $\mu_6$ action, it is necessary that
$\alpha_{i,j}=0$ unless $i\equiv-j\pmod6$. 
Further, $\varphi$ being an $A$-module homomorphism 
means that $\varphi F= F\varphi$ and $\varphi V=V\varphi$.
Let us  now write down what these conditions mean in terms of the ``matrix'' $(\alpha_{i,j})_{i,j}$ of $\varphi$.
Let $m=[\Frr(\mu_{6(q-1)}):\Fp]$, so that $p^mi\equiv i\pmod6$ and
$p^mj\equiv j\pmod{6(q-1)}$ for all $i\in I$ and $j\in J$.  Then, by the results  in the previous subsection, we
have
$$F^m\varphi(e_1)=F^m\left(\sum_{j\equiv-1\pmod6}\alpha_{1,j}f_j\right)
=\sum_{j\equiv-1\pmod6}\sigma^m(\alpha_{1,j})G_jf_j$$
and
$$\varphi F^m(e_1)=\varphi(ue_1)=u\sum_{j\equiv-1\pmod6}\alpha_{1,j}f_j$$
for a certain $u\in W^\times$ (by Equation~\eqref{eq:ordci}).
Equating coefficients of $f_j$ then yields that
$u\alpha_{1,j}=\sigma^m(\alpha_{1,j})G_j$.  
However, we know from Equation~\eqref{eq:ordGj} that $\ord(G_j)=(1/3)m>0$. Hence $\alpha_{1,j}=0$ for all $j\in J$.  
Similarly,  we have
$$V^m\varphi(e_{-1})=V^m\left(\sum_{j\equiv1\pmod6}\alpha_{-1,j}f_j\right)
=\sum_{j\equiv1\pmod6}\sigma^{-m}(\alpha_{-1,j})(p^m/G_j)f_j$$
and
$$\varphi
V^m(e_{-1})=\varphi(ve_{-1})=v\sum_{j\equiv1\pmod6}\alpha_{-1,j}f_j$$
for some $v\in W^\times$ (by Equation~\ref{eq:ordci}).
Equating coefficients of $f_j$ then shows that
$$v\alpha_{-1,j}=\sigma^{-m}(\alpha_{-1,j})(p^m/G_j).$$  
But Equation~\eqref{eq:ordGj} tells us that $\ord(p^m/G_j)=(1/3)m>0$. 
This implies that   $\alpha_{-1,j}=0$ for all $j\in J$.   

Thus every
$\varphi\in\Hom_A\left(H^1(E_0)/p,H^1(C_{6,q})/p\right)^{\mu_6}$
satisfies $\varphi(e_1)=\varphi(e_{-1})=0$.  This proves that
$\Hom_A\left(H^1(E_0)/p,H^1(C_{6,q})/p\right)^{\mu_6}=0$ which
completes the proof of part (1) of the Proposition.
\qed

\subsection{Proof of Proposition~\ref{prop:sha-alg} part (2)}
We now turn to part (2) of the Proposition and assume that
$p\equiv-1\pmod6$.  
For any $n\geq 1$, the set $I\times J$ indexes the eigenspaces of
$\mu_6\times\mu_{6(q-1)}$ acting on  $\Hom(H^1(E_0)/p^n,H^1(C_{6,q})/p^n)$.
And the subset (which we denote by $(I\times J)^{\mu_6}$) indexing invariants under the anti-diagonal action of $\mu_6$ 
consists of pairs $(i,j)$ with $i\equiv-j\pmod6$.  

Define a bijection 
\begin{equation}\label{eq:bijection-p=-1(6)}
(I\times J)^{\mu_6}\to S:=\{1,5\}\times\{1,\dots,q-1\}    
\end{equation}
by $(i,j)\mapsto(b,a)$ where $6a-b\equiv j \pmod{6(q-1)}$ (so that
$b\equiv i\pmod6$).  Under this bijection, $(I_0\times J_1)^{\mu_6}$
corresponds to pairs $(1,a)$ where $0<a<q/6$, and
$(I_1\times J_0)^{\mu_6}$ corresponds to pairs $(5,a)$ where
$5q/6<a<q$.  (See the definitions of $I_0$, $I_1$, $J_0$ and $J_1$ in
Section~\ref{ss:prelims}.)  We thus define
$$S_0=\{(1,a) : 0<a<q/6\}
\quad\text{and}\quad
S_1=\{(5,a) : 5q/6<a<q\}.$$

The action of $\<p\>$ on $I\times J$ preserves $(I\times J)^{\mu_6}$
and so by transport of structure we get a (non-standard) action on
$S$ which we will make explicit below.
Let $O$ be the set of orbits of $\<p\>$ on $S$.  Given an orbit $o\in O$, 
define 
$$d(o):=\min\left(|o\cap S_0|,|o\cap S_1|\right).$$
Part (2) of the proposition will be a consequence of the following
``equidistribution'' result.

\begin{prop}\label{prop:equidistribution}
For every $o\in O$, $|o\cap S_0|=|o\cap S_1|$.
\end{prop}

Indeed,  this proposition implies that
 $$\sum_{o\in O}d(o)=\sum_{o\in O}|o\cap S_0|=|S_0|=\lfloor q/6\rfloor.$$
On the other hand, by equations \eqref{eq:sha-Br} and \eqref{eq:Br-Hom} above and  \cite[Thm.~8.3]{UlmerBS}, recall that
  $$\dim\sha(E)=\sum_{o\in O}d(o).$$
 Hence we have  $\dim\sha(E)=\lfloor q/6 \rfloor$, so that proving Proposition~\ref{prop:equidistribution} 
 will complete the proof of  part (2) of Proposition~\ref{prop:sha-alg}.

\begin{proof}[Proof of Proposition~\ref{prop:equidistribution}]
  We begin the proof by making the action of $\<p\>$ on $S$ more
  explicit.  Suppose that $(i,j)\in (I\times J)^{\mu_6}$ corresponds to $(b,a)\in S$ through the bijection \eqref{eq:bijection-p=-1(6)} and that $p\cdot (i,j) = (pi,pj)$
  corresponds to $(b',a')$.  Then $b'=6-b$ and
  $6a'-b'\equiv p(6a-b)\mod{6(q-1)}$, so that
\begin{align*}
a'&\equiv pa-\frac{p+1}6b+1\pmod{q-1}
\equiv\begin{cases}
pa-\frac{p-5}6\pmod{q-1}&\text{if $b=1$}\\
pa-\frac{5p-1}6\pmod{q-1}&\text{if $b=5$}.
\end{cases}
\end{align*}

We now divide the proof into two cases according to $q\pmod6$.  Suppose first
that $q\equiv1\pmod6$, so that $q=p^f$ with $f$ even.  Then using the
last displayed formula, one finds that $q$ acts on $S$ by
$(b,a)\mapsto(b',a')$ where $b'=b$ and
$$a'\equiv\begin{cases}
a-\frac{q-1}6\pmod{q-1}&\text{if $b=1$}\\
a-\frac{5p-5}6\pmod{q-1}&\text{if $b=5$}.
\end{cases}$$ 
It follows that the orbits of $\<q\>$ have size exactly 6, all
elements of an orbit have the same value of $b$, and each orbit meets
either $S_0$ or $S_1$ in exactly one point and does not meet the
other.  (If the constant value of $b$ is 1, the orbit meets $S_0$ and
if it is $5$, the orbit meets $S_1$.)  The orbits of $\<p\>$ are
unions of an even number of orbits of $\<q\>$, half of them meeting
$S_0$ and half of them meeting $S_1$.  It follows that
$|o\cap S_0|=|o\cap S_1|$ for all orbits $o$ of $\<p\>$.  This
completes the proof in the case when $q\equiv1\pmod6$. 

Now assume that $q\equiv-1\pmod6$, so that $q=p^f$ with $f$ odd.  In
this case, $q$ acts on $S$ by
$(b,a)\mapsto(b',a')$ where $b'=6-b$ and
$$a'\equiv\begin{cases}
a-\frac{q-5}6\pmod{q-1}&\text{if $b=1$}\\
a-\frac{5q-1}6\pmod{q-1}&\text{if $b=5$.}
\end{cases}$$ 
Note that $q$ interchanges the subsets $S_0$ and $S_1$, so every orbit
of $\<q\>$ on $S$ meets $S_0$ and $S_1$ in the same number of points.
Since the orbits of $\<p\>$ are unions of orbits of $\<q\>$, it follows
that the orbits $o$ of $\<p\>$ satisfy $|o\cap S_0|=|o\cap S_1|$.  This
completes the proof in the case $q\equiv-1\pmod6$, and thus in general.
\end{proof}


\section{Archimedean size of $L^*(E)$ and the Brauer--Siegel ratio}\label{s:BS} 
Define the exponential differential height of $E=E_{q,r}$ by $H(E):=r^{\deg(\omega_E)}$. 
As we have seen in Section~\ref{ss:definitions}, one has $H(E)=r^{\lceil q/6\rceil}$.
Following Hindry and Pacheco \cite{HindryPacheco16}, consider the Brauer--Siegel ratio $\BS(E)$ of $E$:
$$\BS(E):=\frac{\log\left(\Reg(E)|\sha(E)|\right)}{\log H(E)}.$$
(By Theorem~\ref{thm:BSD}, $\sha(E)$ is finite so that this quantity makes sense). 
Our goal in this section is to estimate the size of the Brauer--Siegel ratio of $E_{q,r}$ for a fixed $r$ as $q\to\infty$.  Here is the
statement.

\begin{thm}\label{thm:BS}
For a fixed $r$, as $q\to\infty$ runs through powers of $p$, one has  
$$\lim_{q\to\infty} \BS(E_{q,r})=1.$$  
\end{thm}
We will actually prove a slightly more precise estimate: namely, 
$$ \frac{\log\big(\Reg(E)|\sha(E)|\big)}{\log r} =  \frac{q}{6} \left(1+ O\left(\frac{\log\log q}{\log q}\right)\right).$$

Thus for large $q$, the product $\Reg(E)|\sha(E)|$ is of size comparable to $r^{q/6}$.
In the case when $p\equiv -1\bmod{6}$ we already know this fact, at least for large enough $r$ (see Corollary~\ref{cor:p-parts-via-L}(3)).
On the other hand, in the case when $p\equiv 1\bmod{6}$, we know from Proposition~\ref{prop:ord-L-p=1(6)}(4)  that $\Reg(E)=1$, so we deduce that $|\sha(E)|$ is ``large'' (of size comparable to $r^{q/6}$).

We saw in Equation~\eqref{eq:BSD-simplified} that
$$L^*(E)=\frac{\Reg(E)|\sha(E)|}{H(E)r^{-1}}
=\frac{\Reg(E)|\sha(E)|}{r^{\lfloor q/6\rfloor}},$$
so, given the definition of $\BS(E)$, the above theorem will be an immediate consequence of the following one, which is the main result of this section.

\begin{thm}\label{thm:L*-estimate}
For a fixed $r$, as $q\to\infty$ runs through powers of $p$, one has  
$$\lim_{q\to\infty} \frac{\log L^*(E_{q,r})}{q}=0.$$    
\end{thm}

To prove this we estimate $\log L^\ast(E_{q,r})$ from above and from below. While the upper bound is relatively straightforward, proving the required lower bound is more demanding.
Before proving the theorem at the end of this section, we first collect various intermediate results in the next few
subsections.

\subsection{Explicit special value}
Recall from Theorem~\ref{thm:L-elem} that
$$L(E,s)=\prod_{o\in O_{r,6,q}^\times}
\left(1-G(\pi_2(o))^{m_2(o)}G(\pi_3(o))r^{-s|o|}\right)$$ 
where $O_{r,6,q}^\times$ denotes the set of orbits of $\<r\>$  acting on
$(\Z/6\Z)^\times\times\Fq^\times$.  To lighten notation we write
$$\omega(o):=G(\pi_2(o))^{m_2(o)}G(\pi_3(o))$$
for the remainder of the article. Note that $\omega(o)$ is a Weil integer of size $p^{\nu|o|}=r^{|o|}$, 
where a ``Weil integer of size $p^{c}$'' is an algebraic
integer whose absolute value in every complex embedding is $p^{c}$.

We partition $O^\times:=O_{r,6,q}^\times$ as
$O^\times=O^\times_1\cup O^\times_2$ where $O^\times_1$ consists of
those orbits $o$ such that $\omega(o)=r^{|o|}$.  
Thus the orbits in $O^\times_1$ are the ones contributing zeroes to the $L$-function. 
In particular, we have $|O^\times_1|=\rk E(K)$ by our BSD result (Theorem~\ref{thm:BSD}).
From the definition of special value (see Section~\ref{ss:BSD-notations}), it is a simple exercise to see that
\begin{equation}\label{eq:L*}
L^*(E)=\prod_{o\in O^\times_1}|o|
\prod_{o\in O^\times_2}\left(1-\frac{\omega(o)}{r^{|o|}}\right).  
\end{equation}

\subsection{Estimates for orbits}
Let us gather here a few estimates to be used below.  Although we only
need the case $n=6$ in this paper, we work in more generality for
future use.

\begin{lemma}\label{lemma:estimates}
  Let $p$ be a prime number, let $q$ and $r$ be powers of $p$, and let
  $n$ be an integer prime to $p$.  Let
  $S^\times=(\Z/n\Z)^\times\times\Fqtimes$ and let $O^\times$ denote the
  set of orbits of $\<r\>$ on $S^\times$.  Then
 \begin{enumerate} 
 \item $\sum_{o\in O^\times} |o| = |S^\times|=\phi(n)(q-1)$,
 \item $\sum_{o\in O^\times} 1 = |O^\times| \ll q/\log q$,
 \item $\sum_{o\in O^\times} \log |o| \ll q \log\log q / \log q$.
 \end{enumerate}	
 The implied constants depend only on $r$ and $n$.
\end{lemma}

\begin{proof} 
  By general properties of group actions, $S^\times$ decomposes as
  the disjoint union of orbits $o\in O^\times$; this yields (i).
  To prove (ii), we study ``long'' orbits and ``short'' orbits separately.
  Let $x\geq 1$ be a parameter to be chosen later.  Then
    \[
    \left|\big\{ o \in O^\times : |o|> x \big\}\right|
    = \sum_{\substack{ o\in O^\times\\ |o|> x}} 1 
    \leq \sum_{\substack{ o\in O^\times \\ |o|> x}} \frac{|o|}{x}
    \leq \frac{1}{x}\sum_{o\in O^\times} |o| = \frac{|S^\times|}{x}.
    \]
    Let $o\in O^\times$ be the orbit through $(i, \alpha)$. As was
    noted in Section~\ref{ss:orbits}, $|o|\geq [\Frr(\alpha):\Frr]$.
    In particular
    $\left|\big\{ o\in O^\times: |o|\leq x \big\}\right|$ is at most
    $\left|\big\{\alpha\in\Fpbar : [\Frr(\alpha):\Frr]\leq
      x\big\}\right|$.  An element $\alpha\in\Fpbar$ has degree
    $\leq x$ over $\Frr$ if and only if its monic minimal polynomial
    has degree $\leq x$.  The Prime Number Theorem for $\Frr[t]$
    implies that there are at most $c_r r^x/x$ monic irreducible
    polynomials of degree $\leq x$ in $\Frr[t]$ (see
    \cite[Thm.~2.2]{RosenNTFF}) for some constant $c_r>0$ depending at
    most on~$r$.  This argument yields that
    $\left|\big\{ o\in O^\times : |o|\leq x \big\}\right|\leq c_r
    r^x/x$.  Adding the two contributions, and choosing
    $x= \log q/\log r$, we find that $|O^\times| \leq c' q/\log q$
    where $c'$ depends only on $r$ and $n$.

    Let us finally turn to the proof of (iii): given a parameter
    $y\geq 1$, we have
 \begin{align*}
 \sum_{o\in O^\times}\log|o|
 &=\sum_{|o|\leq y} \log|o| + \sum_{|o|> y} \log|o|
 \leq \log y \sum_{|o|\leq y} 1 + \sum_{|o|> y} \frac{\log|o|}{|o|} |o| \\
 &\leq \log y  \sum_{o\in O^\times} 1  + \frac{\log y}{y}\sum_{|o|> y}  |o| 
 \leq \log y  |O^\times| + \frac{\log y}{y} |S^\times|,
 \end{align*}
 because $x\mapsto (\log x)/x$ is decreasing on $(e, \infty)$.  Upon
 using (ii) and choosing $y=\log q$, one finds that
 $\sum_{o\in O^\times}\log|o|\leq c'' q\log\log q/\log q$, where $c''$
 depends only on $r$ and $n$.  This is the desired estimate.
\end{proof}

\subsection{Linear forms in logarithms} 
For the convenience of the reader, we quote a special case of the main
result of \cite{BakerWustholz93} about $\Z$-linear forms in logarithms of algebraic numbers.   Choose once and for all an
embedding $\Qbar\hookrightarrow\C$ 
and fix the branch of the complex logarithm $\log:\C\to\C$  
with the imaginary part of $\log z$ in $(-\pi,\pi]$ for all $z\in\C$.
In particular, if $|z|=1$, then $|\log(z)|\le\pi$ and $\log(-1)=i\pi$.
%
Define the modified height $\hght'_F$ as
follows: For a number field $F$ and $\alpha\in F$, put
\[\hght_F'(\alpha) 
:= \frac{1}{[F:\Q]}\max\left\{\hght_F(\alpha), |\log\alpha|,
  1\right\},\]
where $\hght_F(\alpha)$ denotes the usual logarithmic Weil height of
$\alpha$ (relative to $F$), see \cite[B.2]{HindrySilvermanDG}. 

Let $\alpha_1, \alpha_2$ be two algebraic numbers (not $0$ or $1$) and
denote by $\log\alpha_1$, $\log\alpha_2$ their logarithms.  Let
$F\subset\Qbar$ be the number field generated by $\alpha_1, \alpha_2$
over $\Q$, and let $d:=[F:\Q]$.  Let $B=(b_1,b_2)$ with
$b_1, b_2\in\Z$ not both zero and set
$\hght'(B) := \max\{\hght_\Q(b_1:b_2), 1\},$ where $\hght_\Q$ here
denotes the logarithmic Weil height on $\P^1_\Q$ (relative to
$\Q$). Note that $\hght'(B)\leq \log\max\{|b_1|, |b_2|,e\}$.

With notation as above, let
$\Lambda:=b_1\log\alpha_1+b_2\log\alpha_2\in\C$.  
Then the Baker--W\"ustholz theorem states that either $\Lambda=0$ or
\begin{equation}\label{eq:BW}
\log|\Lambda|> - c_d\hght'_F(\alpha_1)\hght'_F(\alpha_2)\hght'(B)
\end{equation}
where $c_d>0$ is an explicit constant depending only on $d$.

We make use of the Baker--W\"ustholz theorem to prove the following: 
 
\begin{thm} 
Let $p$ be an odd prime number.
Let $z\in\Qbar$ be a Weil integer of size $p^a$, and let
$\zeta\in\Qbar$ be a root of unity.  For any integer $L\neq 0$, either
$\zeta (z p^{-a})^L =1$ or   
 \begin{equation}
 \log\big|1- \zeta(zp^{-a})^L\big| \geq - c_0 - c_1\log|L|,
\end{equation}
 for some effective constants $c_0, c_1>0$ depending at most on $p$,
 $a$, the degree of $z$ over $\Q$, and the order of $\zeta$.  
 \end{thm}

\begin{proof}
  Let $F:=\Q(\zeta, z )$ be the number field generated by $\zeta$ and
  $z$ (viewed as a subfield of $\Qbar$), and $d$ be its degree over
  $\Q$.  We begin by estimating the modified height of
  $z p^{-a}$. 
  By assumption $z$ is a Weil integer of size $p^a$.  Straightforward
  estimates imply that the absolute logarithmic Weil height of
  $z p^{-a}$ is at most $\log p^{a}$. 
Therefore, 
\[\hght'_F(zp^{-a}) 
\leq \max\left\{ \log p^{a}, \frac{|\log (zp^{-a})|}{d}, \frac{1}{d}\right\}
\leq \max\left\{ \log p^{a}, \frac{\pi}{d}\right\}, \]
We have used here that $|zp^{-a}|=1$ in the chosen complex
embedding.  
 
For all $|x|\leq\pi/2$, we have $|\sin x|\geq \frac{2}{\pi} |x|$
and thus, for all $|\theta|\leq\pi$, we have
$$\big|1-\ee^{i \theta}\big| 
= 2\left|\sin \frac{\theta}{2}\right| 
\geq \frac{2}{\pi} |\theta|.$$
If $0 < |\theta|<\pi$, this leads to
$\log\big|1-\ee^{i \theta}\big| \geq \log(2/\pi)+\log|\theta|$.
 	
In the given complex embedding $F\subset\Qbar\hookrightarrow\C$, one
can write $\zeta = \ee^{2\pi ik/n}$ for some $n\in \Z_{\geq 1}$ and
$k\in\{1, \dots, n-1\}$ coprime to $n$ (so that $\zeta$ is a primitive
$n$th root of unity).  There is also a unique angle
$\phi\in(-\pi,\pi]$ such that $zp^{-a}=\ee^{i\phi}$.  Let $L\neq 0$ be
an integer.  To prove the theorem, we may assume that
$\zeta(zp^{-a})^{L}\neq 1$.  Write
$$\zeta(zp^{-a})^{L}=\ee^{i(2\pi k/n+L\phi)}=\ee^{i\tilde\theta}$$
where $\tilde\theta\in(-\pi,\pi]$, and let $m$ be the integer such
that $2\pi k/n+L\phi=2\pi m+\tilde\theta$.  Note that $|m|\le(|L|+3)/2$.
The trigonometric considerations above show that
\begin{align*}
\log\big|1 -\zeta(zp^{-a})^L\big|
 &= \log\big|1 -\ee^{i\tilde\theta}\big|\\
 &\geq \log(2/\pi) + \log|\tilde\theta|\\
&=\log(2/\pi) + \log|2\pi k/n+L\phi-2\pi m|\\
&=\log(2/(n\pi))+\log|2\pi(k-nm)+Ln\phi|.
\end{align*}

Let us now consider 
the $\Z$-linear combination of logarithms of algebraic numbers
$$\Lambda := b_1\log(-1) + b_2\log(zp^{-a}),$$
where $B=(b_1,b_2) := (2(k-mn), nL)\neq(0,0)$.
Note that $\log(-1)=i\pi$ and $\log(zp^{-a}) = i\phi$, so that
$\Lambda = i(2\pi(k-nm)+Ln\phi)$.
By assumption, $\Lambda\neq0$ so the Baker--W\"ustholz theorem \eqref{eq:BW}
yields that 
\[\log|\Lambda| \geq -c_d \hght'_F(-1)\hght'_F(zp^{-a})\hght'(B).\]
As was shown above, 
$$\hght'_F(zp^{-a})\leq \max\{\log p^{a},\pi/d\},$$
and one can easily see that $\hght'_F(-1)=\pi/d$. 
Also,  $\hght'(B)\leq \log\max\{|b_1|, |b_2|,\ee\}$, where
$$|b_1|=|2(k-mn)|\le 2n(1+|m|)\le(3+|L|)\le 3n|L|,$$
and $|b_2|=n|L|$, so that $\hght'(B)\le\log(3n|L|)$.

Putting  these estimates together, we arrive at
\begin{align*}
\log\big|1 -\zeta  (z p^{-a})^L\big| 
&\geq \log \frac{2}{n\pi} -c_d \frac\pi d\max\left\{\log
  p^a,\frac{\pi}{d} \right\}\log(3n|L|)\\
&\geq -c_0-c_1\log |L|
\end{align*}
where $c_0$ and $c_1$ are certain positive constants depending only on $p$, $a$, $n$, and $d$.  This
completes the proof of the theorem.
 \end{proof}

We now apply this result to the  situation at hand.
For any orbit $o\in O^\times_{r,6,q}$, we deduce from Proposition~\ref{prop:G-power} that we can write $G(\pi_2(o)) = \zeta_2 g_2^{|\pi_2(o)|\nu}$ where $\zeta_2=\pm1$, and $g_2\in\Q(\mu_{2p})$ is a Weil integer of size $p^{1/2}$.
Similarly, letting $c$ be the order of $p$ modulo $3$, Proposition~\ref{prop:G-power} implies that $G(\pi_3(o)) = \zeta_3 g_3^{|\pi_3(o)|\nu/c}$ where
$\zeta_3$ is a $3$rd root of unity and $g_3\in\Q(\mu_{3p})$ is a Weil integer of size $p^{c/2}$. 
Since $m_2(o)|\pi_2(o)|=|o|$ and $|\pi_3(o)| = |o|$, and since $c\in\{1,2\}$, we find that 
$$\omega(o) = \zeta_2^{m_2(o)}\zeta_3  (g_2^2g_3^{2/c})^{|o|\nu/2}.$$ 
For any orbit $o\in O^\times$, it follows that $\omega(o)$ is of the form $\omega(o)=\zeta_o g_o^{|o|\nu/2}$ where $\zeta_o = \zeta_2^{m_2(o)}\zeta_3$ is a 6-th root of unity and
$g_o=g_2^2g_3^{2/c}\in\Q(\mu_{6p})$ is a Weil integer of size $p^2$.

Using the previous theorem for $\zeta=\zeta_o$, $z=g_o$ (with $a=2$) and $L= |o|\nu/2$, and setting $c_2=c_0+c_1\log(\nu/2)$, 
one obtains the following corollary: 

\begin{cor}\label{cor:BaWu-application}
For any orbit $o\in O_{r,6,q}^\times$, either $\omega(o)/r^{|o|}=1$  \textup{(}i.e., $o\in O^\times_1$\textup{)} or
$$\log\left|1-\frac{\omega(o)}{r^{|o|}}\right|
\geq -c_2-c_1\log|o|.$$
\end{cor}

\subsection{Proof of Theorem~\ref{thm:L*-estimate}}
Recall that the theorem asserts that
$$\lim_{q\to\infty}\frac{\log L^*(E_{q,r})}q=0.$$
We saw in Equation~\eqref{eq:L*} that
$$L^*(E_{q,r})=\prod_{o\in O^\times_1}|o|
\prod_{o\in O^\times_2}\left(1-\frac{\omega(o)}{r^{|o|}}\right)$$
where $O^\times_1\subset O^\times_{r,6,q}$
consists of those orbits $o$ such that $\omega(o)=r^{|o|}$ and $O^\times_2 = O^\times_{r,6,q}\smallsetminus O_1^\times$. 

It is clear that $|1-\omega(o)/r^{|o|}|\le2$ for all $o\in O^\times$. 
We can thus bound $\log L^*(E)$ from above as
follows:
\begin{align*}
\log L^*(E_{q,r})
&=\log\left(  \prod_{o\in O^\times_1}|o|\prod_{o\in O^\times_2}
   \left(1-\frac{\omega(o)}{r^{|o|}}\right)\right)
\le \sum_{o\in O^\times_1}\log|o|+\sum_{o\in O^\times_2}\log2\\
&\ll \frac{q\log\log q}{\log q}+\frac q{\log q}\log 2 
\ll \frac{q\log\log q}{\log q} 
\end{align*}
where we made use of Lemma~\ref{lemma:estimates} in the  last step.  Thus
$$\limsup_{q\to\infty}\frac{\log L^*(E_{q,r})}q\ll
\limsup_{q\to\infty}\left(\frac{\log\log q}{\log q}\right)=0.$$

We now turn to a lower bound.  We obtain from
Corollary~\ref{cor:BaWu-application} that
\begin{align*}
\log L^*(E_{q,r})
&=\log\left(  \prod_{o\in O^\times_1}|o|\prod_{o\in O^\times_2}
   \left(1-\frac{\omega(o)}{r^{|o|}}\right)\right)
\ge \sum_{o\in O^\times_1}\log|o|
  +\sum_{o\in O^\times_2}(-c_2-c_1\log|o|)\\
&\gg-\frac q{\log q}-\frac{q\log\log q}{\log q} 
\gg-\frac{q\log\log q}{\log q} 
\end{align*}
 using Lemma~\ref{lemma:estimates} again for the penultimate inequality.  Therefore
$$\liminf_{q\to\infty}\frac{\log L^*(E_{q,r})}q\gg
\liminf_{q\to\infty}\left(-\frac{\log\log q}{\log q}\right)=0.$$

Combining the upper and lower bounds, we finally obtain that
$$\lim_{q\to\infty}\frac{\log L^*(E_{q,r})}{q}=0,$$
and this completes the proof of Theorem~\ref{thm:L*-estimate}.
\qed

As a direct consequence of Corollary~\ref{cor:p-parts-via-L}(1) and Theorem~\ref{thm:BS}, we obtain the following.
\begin{cor}\label{cor:BS-p=1(6)} 
Assume that $p\equiv 1\pmod 6$. Then, as $q\to\infty$, we have $|\sha(E)[p^\infty]|=1$ and
$$ |\sha(E)| \geq H(E)^{1+o(1)} = r^{\frac{q}{6}(1+o(1))}.$$
\end{cor}

\bibliography{database}

\end{document}

%% file: formatting.tex
\numberwithin{equation}{section}

\swapnumbers
\theoremstyle{plain}
\newtheorem{thm}[subsection]{Theorem}

\newtheorem{prop}[subsection]{Proposition}
\newtheorem{propss}[subsubsection]{Proposition}
\newtheorem{lemma}[subsubsection]{Lemma}
\newtheorem{cor}[subsection]{Corollary}
\newtheorem{corss}[subsubsection]{Corollary}

\theoremstyle{definition}

\theoremstyle{remark}



%% file: macros.tex
\usepackage[OT2,T1]{fontenc}
\DeclareSymbolFont{cyrletters}{OT2}{wncyr}{m}{n}
\DeclareMathSymbol{\sha}{\mathalpha}{cyrletters}{"58}

\newcommand{\CC}{\mathcal{C}}

\newcommand{\DD}{\mathcal{D}}

\renewcommand{\SS}{\mathcal{S}}

\newcommand{\WW}{{\mathcal{W}}}

\newcommand{\EE}{\mathcal{E}}

\newcommand{\NN}{\mathcal{N}}

\newcommand{\F}{\mathbb{F}}
\newcommand{\Fp}{{\mathbb{F}_p}}

\newcommand{\Fq}{{\mathbb{F}_q}}

\newcommand{\Fqtimes}{{\mathbb{F}_q^\times}}

\newcommand{\Fpbar}{{\overline{\mathbb{F}}_p}}

\newcommand{\Qbar}{{\overline{\mathbb{Q}}}}
\newcommand{\Zbar}{{\overline{\mathbb{Z}}}}

\newcommand{\ratto}{{\dashrightarrow}}

\newcommand{\Qlbar}{{\overline{\mathbb{Q}}_\ell}}

\newcommand{\Qpbar}{{\overline{\mathbb{Q}}_p}}

\newcommand{\Z}{\mathbb{Z}}

\newcommand{\Q}{\mathbb{Q}}

\newcommand{\C}{\mathbb{C}}
\newcommand{\A}{\mathbb{A}}
\renewcommand{\P}{\mathbb{P}}

\newcommand{\K}{K}

\newcommand{\<}{\langle}
\renewcommand{\>}{\rangle}

\newcommand{\tensor}{\otimes}
\newcommand{\compose}{\circ}
\newcommand{\sdp}{{\rtimes}}
\newcommand{\nodiv}{\not|}
\def\nodiv{\mathrel{\mathchoice{\not|}{\not|}{\kern-.2em\not\kern.2em|}
{\kern-.2em\not\kern.2em|}}}

\DeclareMathOperator{\Tr}{Tr}

\DeclareMathOperator{\N}{N}

\DeclareMathOperator{\ord}{ord}
\DeclareMathOperator{\rk}{Rank}

\DeclareMathOperator{\Hom}{Hom}

\DeclareMathOperator{\Reg}{Reg}

\DeclareMathOperator{\Br}{Br}

\DeclareMathOperator{\lcm}{lcm}

\DeclareMathOperator{\Fr}{Fr}

\def\clap#1{\hbox to 0pt{\hss#1\hss}}

%% file: ST.bbl
\def\cprime{$'$}
\begin{bibdiv}
\begin{biblist}

\bib{BakerWustholz93}{article}{
      author={Baker, A.},
      author={W\"{u}stholz, G.},
       title={Logarithmic forms and group varieties},
        date={1993},
     journal={J. Reine Angew. Math.},
      volume={442},
       pages={19\ndash 62},
}

\bib{CohenNT1}{book}{
      author={Cohen, H.},
       title={Number theory. {V}ol. {I}. {T}ools and {D}iophantine equations},
      series={Graduate Texts in Mathematics},
   publisher={Springer},
     address={New York},
        date={2007},
      volume={239},
}

\bib{Gordon79}{article}{
      author={Gordon, W.~J.},
       title={Linking the conjectures of {A}rtin--{T}ate and
  {B}irch--{S}winnerton-{D}yer},
        date={1979},
     journal={Compositio Math.},
      volume={38},
       pages={163\ndash 199},
}

\bib{Griffonpp1801}{article}{
      author={Griffon, R.},
       title={Bounds on special values of {$L$}-functions of elliptic curves in
  an {A}rtin--{S}chreier family},
        date={2018},
        note={Preprint, arXiv:1801.08492. To appear in the {\it European
  Journal of Mathematics}},
}

\bib{HindryPacheco16}{article}{
      author={Hindry, M.},
      author={Pacheco, A.},
       title={An analogue of the {B}rauer--{S}iegel theorem for abelian
  varieties in positive characteristic},
        date={2016},
     journal={Mosc. Math. J.},
      volume={16},
       pages={45\ndash 93},
}

\bib{HindrySilvermanDG}{book}{
      author={Hindry, M.},
      author={Silverman, J.~H.},
       title={Diophantine geometry},
      series={Graduate Texts in Mathematics},
   publisher={Springer-Verlag, New York},
        date={2000},
      volume={201},
        note={An introduction},
}

\bib{Katz81}{incollection}{
      author={Katz, N.~M.},
       title={Crystalline cohomology, {D}ieudonn\'{e} modules, and {J}acobi
  sums},
        date={1981},
   booktitle={Automorphic forms, representation theory and arithmetic
  ({B}ombay, 1979)},
      series={Tata Inst. Fund. Res. Studies in Math.},
      volume={10},
   publisher={Tata Inst. Fundamental Res., Bombay},
       pages={165\ndash 246},
}

\bib{Kleiman68}{incollection}{
      author={Kleiman, S.~L.},
       title={Algebraic cycles and the {W}eil conjectures},
        date={1968},
   booktitle={Dix expos\'{e}s sur la cohomologie des sch\'{e}mas},
      series={Adv. Stud. Pure Math.},
      volume={3},
   publisher={North-Holland, Amsterdam},
       pages={359\ndash 386},
}

\bib{KatoTrihan03}{article}{
      author={Kato, K.},
      author={Trihan, F.},
       title={On the conjectures of {B}irch and {S}winnerton-{D}yer in
  characteristic {$p>0$}},
        date={2003},
     journal={Invent. Math.},
      volume={153},
       pages={537\ndash 592},
}

\bib{Milne75}{article}{
      author={Milne, J.~S.},
       title={On a conjecture of {A}rtin and {T}ate},
        date={1975},
     journal={Ann. of Math. (2)},
      volume={102},
       pages={517\ndash 533},
}

\bib{Neron65}{article}{
      author={N\'eron, A.},
       title={Quasi-fonctions et hauteurs sur les vari\'et\'es ab\'eliennes},
        date={1965},
     journal={Ann. of Math. (2)},
      volume={82},
       pages={249\ndash 331},
}

\bib{PriesUlmer16}{article}{
      author={Pries, R.},
      author={Ulmer, D.},
       title={Arithmetic of abelian varieties in {A}rtin--{S}chreier
  extensions},
        date={2016},
     journal={Trans. Amer. Math. Soc.},
      volume={368},
       pages={8553\ndash 8595},
}

\bib{RosenNTFF}{book}{
      author={Rosen, M.},
       title={Number theory in function fields},
      series={Graduate Texts in Mathematics},
   publisher={Springer-Verlag},
     address={New York},
        date={2002},
      volume={210},
}

\bib{Shioda86}{article}{
      author={Shioda, T.},
       title={An explicit algorithm for computing the {P}icard number of
  certain algebraic surfaces},
        date={1986},
     journal={Amer. J. Math.},
      volume={108},
       pages={415\ndash 432},
}

\bib{Shioda92}{article}{
      author={Shioda, T.},
       title={Some remarks on elliptic curves over function fields},
        date={1992},
     journal={Ast\'{e}risque},
      number={209},
       pages={99\ndash 114},
        note={Journ\'{e}es Arithm\'{e}tiques, 1991 (Geneva)},
}

\bib{SilvermanAEC}{book}{
      author={Silverman, J.~H.},
       title={The arithmetic of elliptic curves},
     edition={Second},
      series={Graduate Texts in Mathematics},
   publisher={Springer},
     address={Dordrecht},
        date={2009},
      volume={106},
}

\bib{SilvermanAT}{book}{
      author={Silverman, J.~H.},
       title={Advanced topics in the arithmetic of elliptic curves},
      series={Graduate Texts in Mathematics},
   publisher={Springer-Verlag},
     address={New York},
        date={1994},
      volume={151},
}

\bib{Tate66b}{incollection}{
      author={Tate, J.~T.},
       title={On the conjectures of {B}irch and {S}winnerton-{D}yer and a
  geometric analog},
        date={1966},
   booktitle={S{\'e}minaire bourbaki, vol.\ 9},
   publisher={Soc. Math. France},
     address={Paris},
       pages={Exp.\ No.\ 306, 415\ndash 440},
}

\bib{Ulmer02}{article}{
      author={Ulmer, D.},
       title={Elliptic curves with large rank over function fields},
        date={2002},
     journal={Ann. of Math. (2)},
      volume={155},
       pages={295\ndash 315},
}

\bib{Ulmer07b}{article}{
      author={Ulmer, D.},
       title={{$L$}-functions with large analytic rank and abelian varieties
  with large algebraic rank over function fields},
        date={2007},
     journal={Invent. Math.},
      volume={167},
       pages={379\ndash 408},
}

\bib{Ulmer11}{incollection}{
      author={Ulmer, D.},
       title={Elliptic curves over function fields},
        date={2011},
   booktitle={Arithmetic of ${L}$-functions ({P}ark {C}ity, {UT}, 2009)},
      series={IAS/Park City Math. Ser.},
      volume={18},
   publisher={Amer. Math. Soc.},
     address={Providence, RI},
       pages={211\ndash 280},
}

\bib{UlmerBS}{unpublished}{
      author={Ulmer, D.},
       title={On the {B}rauer--{S}iegel ratio for abelian varieties over
  function fields},
        date={2018},
        note={Preprint, arXiv:1806:01961},
}

\bib{WashingtonCF}{book}{
      author={Washington, L.~C.},
       title={Introduction to cyclotomic fields},
     edition={Second},
      series={Graduate Texts in Mathematics},
   publisher={Springer-Verlag},
     address={New York},
        date={1997},
      volume={83},
}

\end{biblist}
\end{bibdiv}
